\documentclass{amsart}
\usepackage{amsmath}
\usepackage{amsopn}
\usepackage{amssymb}
\usepackage{latexsym}
\usepackage{amsfonts}
\usepackage{amsthm}
\usepackage[all]{xy}
\usepackage{amssymb,amsmath,amscd,amsthm,epsfig, tikz}
\newcommand{\then}{\Longrightarrow}

\newcommand{\inver}{{\rm inv}}
\newcommand{\inv}{{\rm inv}}
\newcommand{\Inv}{{\rm Inv}}
\newcommand{\Neg}{{\rm Neg}}
\newcommand{\Pos}{{\rm Pos}}
\newcommand{\n}{{\rm n}}

\newcommand{\M}{{\rm M}}

\newcommand{\nega}{{\rm neg}}
\newcommand{\D}{{\rm Des}}
\newcommand{\des}{{\rm des}}
\newcommand{\wdes}{{\rm {wdes}}}

\newcommand{\ZZ}{\mathbb{Z}}

\newtheorem{thm}{Theorem}[section]

\newtheorem{proposition}[thm]{Proposition}
\newtheorem{lemma}[thm]{Lemma}

\newtheorem{corollary}[thm]{Corollary}

\newtheorem{problem}[thm]{Problem}

\theoremstyle{definition}

\newtheorem{rem}[thm]{Remark}
\newtheorem{exa}[thm]{Example}
\newtheorem{defn}[thm]{Definition}

\newtheorem{conj}[thm]{Conjecture}

\newcommand{\een}{\end{enumerate}}
\newcommand{\blem}{\begin{lem}}
\newcommand{\elem}{\end{lem}}
\newcommand{\bcl}{\begin{cla}}
\newcommand{\ecl}{\end{cla}}
\newcommand{\ethm}{\end{thm}}
\newcommand{\bpr}{\begin{pro}}
\newcommand{\epr}{\end{pro}}
\newcommand{\bco}{\begin{cor}}
\newcommand{\eco}{\end{cor}}
\newcommand{\bcon}{\begin{conj}}
\newcommand{\econ}{\end{conj}}
\newcommand{\bde}{\begin{defn}}
\newcommand{\ede}{\end{defn}}
\newcommand{\bex}{\begin{exa}}
\newcommand{\eexa}{\end{exa}}
\newcommand{\bobs}{\begin{obs}}
\newcommand{\eobs}{\end{obs}}
\newcommand{\bexe}{\begin{exe}}
\newcommand{\eexe}{\end{exe}}

\newcommand{\finv}{{\rm finv}}
\newcommand{\hB}{{\widehat B}}%{{\tilde B}}
\newcommand{\hN}{{\widehat N}}%{{\tilde N}}

\begin{document}

\title{Flag Weak Order on Wreath Products}
\bibliographystyle{acm}
%\author{Ron M. Adin, Francesco Brenti, Yuval Roichman}

\author{Ron M.\ Adin}
\address{Department of Mathematics\\
Bar-Ilan University\\
52900 Ramat-Gan\\
Israel} \email{radin@math.biu.ac.il}

\author{Francesco Brenti}
\address{Dipartimento di Matematica, Universit\'{a} di Roma ``Tor
Vergata'', Via della Ricerca Scientifica, 00133 Roma, Italy}
\email{brenti@mat.uniroma2.it}

\author{Yuval Roichman}
\address{Department of Mathematics\\
Bar-Ilan University\\
52900 Ramat-Gan\\
Israel} \email{yuvalr@math.biu.ac.il}

\begin{abstract}
%A natural problem is to give a ``correct definition" of a weak
%order on the wreath product $\ZZ_r\wr S_n$. 
A generating set for the wreath product
$\ZZ_r\wr S_n$ which leads to a nicely behaved weak order is
presented, and properties of the resulting order are studied.
\end{abstract}

\date{submitted October 2, '11; revised September 23, '12}

\thanks{Work of the first and third authors was supported in part by an Internal Research Grant 
from the Office of the Rector, Bar-Ilan Universty}

\maketitle

\section{Introduction}

The weak order on a Coxeter group is a fundamental tool in the
study of the combinatorial structure of this group. A
%long standing open
natural problem is to give a ``correct definition" of a weak order
on the wreath product $G(r,n):=\ZZ_r\wr S_n$. The weak order on a
Coxeter group  is determined via the generating set of simple
reflections and the associated length
function.  % play a crucial role.
 In this paper we address the basic question:
%now arises,
Which generating set for the wreath product is the
counterpart of the set of simple reflections? Unfortunately, the
natural analogue --- the set of complex reflections --- %in a wreath product
does not lead to a nicely behaved partial order.
%The classical Coxeter generating set of $B_n$ does not reflect the
%fact that $B_n$ is a semi-direct product of $\ZZ_2^n$ with $S_n$.
%As a result, the classical weak order on $B_n$ does not include
%the Boolean poset as a subposet, although $\ZZ_2^n$ is a normal
%subgroup (essentially the torus) of $B_n$.
%In this note we suggest a generating set for $B_n$ (and, in general, for $G(r,n)$), which gives a nicely behaved weak order.
%The length statistic plays a crucial role in the definition of this order on classical Coxeter groups.
It will be shown that there is %a proper choice of 
a generating set %for $B_n$, and, in general,
yielding an order on $G(r,n)$ with  properties analogous to those of
the weak order on %the symmetric group 
$S_n = G(1,n)$: 
The resulting poset is a ranked by the Foata-Han flag inversion number; 
it is a self-dual
%orthocomplemented
lattice; it has a Tits-type property; and
its intervals have the desired homotopy types. 
Finally, the associated M\"obius function and relevant generating functions will be computed.

\medskip

The rest of the paper is organized as follows. Necessary
preliminaries and notation are given in Section 2. For the sake of
a clarity, results are first stated and proved for %the case $r=2$
the hyperoctahedral group $B_n=G(2,n)$: 
The generating set and corresponding presentation are described in Section 3, 
the flag weak order is defined in Section 4, 
and its properties are studied in Sections 4-6. 
The corresponding results for %Extensions to the case of a
general $r$ are discussed in Section 7. 
Section 8 contains final remarks and open problems. % are presented in Section 8.

 %plays a role in a generalization of the weak order to wreath products.

\section{Preliminaries}\label{s.prelim}

Let $(W,S)$ be a Coxeter system; thus $W$ is a group with a set of generators
$S=\{s_0,s_1,\ldots,s_n\}$ and a presentation of the form
\[
%\begin{aligned}
W=\langle s_0, s_1, \ldots ,s_n \,|\, (s_i s_j)^{m_{ij}}=e \ (0 \le i \le j \le n) \rangle,
%\end{aligned}
\]
where $m_{ij} =m_{ji} \in \{2,3,\ldots\} \cup \{\infty\}$ and
$m_{ii}=1$.

The {\it (right) weak order} $\leq$ on $W$  is the reflexive and
transitive closure of the relation
$$w \lessdot ws\ \ \ \ \  \Longleftrightarrow\ \ \ \ \  w\in W, \ \ \ \  s \in S\ \ \
\text{ and }\ \ \ \ell(w)+1=\ell(ws),$$ where $\ell(\cdot)$ is the
standard length function with respect to the Coxeter generating set $S$.
The left weak order %$\leq_{RW}$
is defined similarly, with $sw$ instead of $ws$.
For combinatorial and other properties of the weak order the reader is referred to~\cite{BB}.

Let $S_n$ be the symmetric group on the letters $[n]:=\{1,\dots,n\}$.
Recall that $S_n$ is a Coxeter group with
respect to the set of Coxeter generators $S:=\{s_i\,|\,1\le i\le
n-1\}$, where $s_i$ may be interpreted as the adjacent
transposition $(i,i+1)$.

For $\pi\in S_n$ let the {\em inversion set} be
$\Inv(\pi):=\{(i,j):\ i<j, \ \pi(i)>\pi(j)\}$, the {\em inversion
number} be $\inv(\pi):=\#\Inv(\pi)$, and the {\em descent set} be
$\D(\pi):=\{i\in [n-1]:\ \pi(i)>\pi(i+1)\}$. Recall the classical
combinatorial interpretations of the Coxeter length function and
of the (right) weak order~\cite[Cor. 1.5.2, Prop. 3.1.3]{BB}:
\begin{equation}\label{e.inv}
\ell(\pi) = \inv(\pi) = \inv(\pi^{-1}),\qquad
\pi \le \sigma \Longleftrightarrow \Inv(\pi^{-1}) \subseteq \Inv(\sigma^{-1}).
\end{equation}

Let $B_{n}$ be the group of all bijections $\sigma$ of the set
$[\pm n] := \{-n, \ldots, -1, 1, \ldots, n\}$ onto itself such that
\[
\sigma(-a) = -\sigma(a) \qquad(\forall a\in[\pm n]),
\]
with composition as the group operation. $B_n$ is %usually
known as the group of ``signed permutations'' on $[n]$, or as the
{\em hyperoctahedral group} of rank $n$. We identify $S_{n}$ as a
subgroup of $B_{n}$, and $B_{n}$ as a subgroup of $S_{2n}$, in the
natural ways.

For $\sigma \in B_{n}$  let $\Neg(\sigma) := \{ i \in [n]\,: \, \sigma(i)<0 \}$,
$\nega(\sigma) := \#\Neg(\sigma)$ and
$|\sigma| = [|\sigma(1)|, \ldots ,|\sigma(n)|] \in S_n$.

More generally, consider the wreath product $    = \ZZ_r \wr S_n$, where
$\ZZ_r$ is the (additive) cyclic group of order $r$:
\[
G(r,n) := \{g = ((c_1,\ldots,c_n), \sigma)\,|\,\,  %grn
c_i\in\ZZ_r\,(\forall i),\, \sigma\in S_n\},  %grn
\]
with the group operation
\[
%((c_1,\ldots,c_n), \sigma) \cdot ((d_1,\ldots,d_n), \tau) :=  %grn
%((c_1 + d_{\sigma^{-1}(1)}, \ldots, c_n + d_{\sigma^{-1}(n)}), \sigma\tau).  %grn
((c_1,\ldots,c_n), \sigma) \cdot ((d_1,\ldots,d_n), \tau) :=
((c_{\tau(1)} + d_{1}, \ldots, c_{\tau(n)} + d_{n}), \sigma\tau).
\]
(This definition is slightly non-standard, and is chosen for compatibility with the case $r=2$; see below.)
The elements of $\ZZ_r\wr S_n$ may be interpreted as $r$--colored permutations,
i.e., bijections $g$ of the set $\ZZ_r \times [n]$ onto itself such that
\[
g(c,i) = (d,j) \Longrightarrow g(c+c',i) = (d+c',j) \qquad(\forall c, c', d \in \ZZ_r,\, i, j \in [n]).
\]
For example, $G(1,n)$ is naturally isomorphic to the symmetric group $S_n$
and $G(2,n)$ is isomorphic to the hyperoctahedral group $B_n$,
where $ ((c_1, \ldots ,c_n), \sigma) \in G(2,n)$ corresponds to the element $g \in B_n$
such that
\[
g(i) = (-1)^{c_i} \sigma(i) \qquad(\forall i \in [n]).
\]
Thus $\Neg(g) = \{i\,:\, c_i = 1\}$, a basic compatibility (for $r=2$) that
underlies the choice of group operation in $G(r,n)$ above.
Informally, this means that the colors (or signs) $c_i$ are attached
{\em before} the permutation $\sigma$ is applied.

In the special cases $r = 1, 2$, $G(r,n)$ is of course a Coxeter group.

%\todo{Align notation of wreath product}

\medskip

For an $r$-colored permutation $\pi=((c_1,\dots,c_n), \sigma) \in G(r,n)$
let $|\pi| := \sigma$ and $\n(\pi):= \sum\limits_{i=1}^n c_i \in \ZZ$,
where elements of $\ZZ_r$ are interpreted as
the corresponding elements of $\{0, \ldots, r-1\} \subseteq \ZZ$.
Note that, for $r=2$,\ $\n(\pi)=\nega(\pi)$.

\medskip

%A problem, which was posed by Dominique Foata two decades ago, was
%to find the counterparts of classical permutation statistics in
%$r$--colored permutations and words.
The classical inversion number on permutations has a counterpart for wreath products,
the {\em flag inversion number}. It was introduced by Foata and Han~\cite{FH1, FH} and
further investigated in~\cite{Fire, FR}.
%, who showed that this statistics plays a role analogous to the classical inversion number.

\begin{defn}\label{finv}
The {\it flag inversion number} of an $r$-colored permutation $\pi \in G(r,n)$ is defined as
%\begin{equation}\label{finv}
$$
\finv(\pi):=r\cdot\inv (|\pi|)+ \n(\pi).
$$
\end{defn}

\medskip

For a positive integer $m$ and an indeterminate $q$ denote
\[
[m]_q := \frac{q^m-1}{q-1}.
\]

\begin{proposition}\label{finv-gf}\cite[Theorem 7.4]{Fire}
For every $r$ and $n$
$$
\sum\limits_{\pi\in G(r,n)} q^{\finv(\pi)}=\prod\limits_{i=1}^n
[ri]_q.
$$
\end{proposition}

\medskip

\section{Generators and Presentations} %{Two Group Presentations}

The {\em alternating subgroup} of a reflection group is the kernel of
the {\em sign homomorphism} which maps all the Coxeter generators (simple reflections) to $-1$.

%\bigskip

\begin{proposition}\label{t.relation2}
The alternating subgroup of the hyperoctahedral group $B_n=G(2,n)$ is
isomorphic to the abstract group generated by $\{a_i\,:\, 1 \le i \le n-1\}$ with
defining relations
\[
a_i^4=1 \qquad (1 \le i \le n-1), \leqno(A1)
\]
\[
a_i a_j = a_j a_i \qquad (|i-j|>1), \leqno(A2)
\]
\[
a_i a_{i+1} a_i= a_{i+1} a_i a_{i+1} \qquad (1 \le i \le n-1)  \leqno(A3)
\]
and
\[
(a_i a_{i+1})^3=1 \qquad (1 \le i \le n-1). \leqno(A4)
\]
%\[
%a_i^2 a_j^2 =a_j^2 a_i^2 \qquad (1\le i< j\le n)
%\]
\end{proposition}

\begin{proof}
Denote by $B_n^+$ the alternating subgroup of $B_n$,
and let $\hB_n^+$ be the abstract group with the above presentation.
Define a map $\phi$ from the {\em free group} generated by $a_1, \ldots, a_{n-1}$
to $B_n^+$ by
\[
\phi(a_i):= [1, \ldots, -(i+1), i, \ldots,n] \qquad (1\le i \le
n-1).
\]
%We claim that $\phi$, when extended multiplicatively, determines
%a well-defined group isomorphism  $\phi: \hB_n^+ \to B_n^+$.
Since $\phi(a_i) = (i,i+1)(i,-i)$ is a product of two
reflections in $B_n$, it indeed belongs to $B_n^+$. It is easy to
check that relations $(A1)-(A4)$ are satisfied when each $a_i$ is
replaced by $\phi(a_i)$. This therefore defines a group
homomorphism, which we again denote by $\phi$, from $\hB_n^+$ to
$B_n^+$. We shall show that it is actually an isomorphism.

Now, $B_n^+$ is generated by the set $\{(i,i+1)(1,-1)\,:\,1 \le i
\le n-1\}$; see, e.g., \cite[\S 5.1, Exercise 1]{Humphreys}. Since
$\phi(a_i)^2 = (i+1, -(i+1))(i, -i)$, it follows that
\[
\phi(a_i) \phi(a_{i-1})^2 \phi(a_{i-2})^2 \cdots \phi (a_1)^2  =
(i,i+1)(1,-1)
\]
for $1 \le i \le n-1$, and therefore $\phi : \hB_n^+ \to B_n^+$ is surjective.

It remains to show that $\phi$ is injective. Since it is surjective and
$\# B_n^+ = 2^{n-1} n!$, it suffices to show that $\# \hB_n^+ \le 2^{n-1} n!$.

Let $\hN_n^+$ be the subgroup of $\hB_n^+$ generated by $a_1^2, \ldots, a_{n-1}^2$.
We shall show that $\hN_n^+$ is a commutative normal subgroup of $\hB_n^+$.
Indeed, $(A4)$ can be written as
\[
a_i a_{i+1} a_i a_{i+1} a_i a_{i+1} = 1
\]
or, using $(A3)$, as
\[
a_i a_{i+1} a_i a_i a_{i+1} a_i = 1.
\]
Rearrangement gives
\[
a_i^2 a_{i+1} = a_{i+1}^{-1} a_i^{-2}
\]
or
\begin{equation}\label{e.phi1}
a_{i+1}^{-1} a_i^2 a_{i+1} = a_{i+1}^{-2} a_i^{-2} \in \hN_n^+.
\end{equation}
Similarly, $(A4)$ and $(A3)$ for $i-1$ imply
\[
a_{i-1}^{-1} a_i^2 a_{i-1} = a_{i-1}^{-2} a_i^{-2} \in \hN_n^+.
\]
Finally, by $(A2)$,
\[
a_{j}^{-1} a_i a_{j} = a_i \qquad(|i-j| > 1)
\]
so that
\[
a_{j}^{-1} a_i^2 a_{j} = a_i^{2} \in \hN_n^+ \qquad(|i-j| > 1).
\]
Thus $\hN_n^+$ is a normal subgroup of $\hB_n^+$.

Commutativity of $\hN_n^+$ is also easy:
(\ref{e.phi1}) and $(A1)$ imply that
\[
a_{i+1}^{-1} a_i^2 a_{i+1} = a_{i+1}^{-2} a_i^{2}
\]
or
\[
a_{i+1} a_i^2 a_{i+1} = a_i^{2},
\]
so that also
\[
a_{i+1}^{2} a_i^2 a_{i+1}^2 = a_{i+1} (a_{i+1} a_i^2 a_{i+1}) a_{i+1} = a_{i+1} a_i^2 a_{i+1} = a_i^{2}.
\]
Thus, again by $(A1)$,
\[
a_i^2 a_{i+1}^2 = a_{i+1}^{-2} a_i^{2} = a_{i+1}^{2} a_i^{2},
\]
i.e., $a_i^2$ and $a_{i+1}^2$ commute.
This is certainly also the case for $a_i^2$ and $a_j^2$ when $|i-j| > 1$, so
$\hN_n^+$ is commutative.

We can now wrap up the proof:
$\hN_n^+$ is a commutative group generated by the involutions $a_1^2, \ldots, a_{n-1}^2$.
Thus each element of $\hN_n^+$ can be written as a product $a_{i_1}^2 \cdots a_{i_k}^2$
for some $k \ge 0$ and $1\le i_1 < \ldots < i_k \le n-1$. In particular, $\# \hN_n^+ \le 2^{n-1}$.
Also, $\hN_n^+$ is a normal subgroup of $\hB_n^+$.
The quotient $\hB_n^+ / \hN_n^+$ is generated by ${\bar a}_i$, the cosets corresponding to
the generators $a_i$ of $\hB_n^+$ $(1 \le i \le n-1)$.
The ${\bar a}_i$ satisfy the same relations $(A2) - (A4)$ as the $a_i$, with $(A1)$ replaced by
\[
{\bar a}_i^2 = 1 \qquad (1 \le i \le n-1).
\]
These are exactly the Coxeter relations defining the symmetric group $S_n$
(actually, $(A3)$ is now equivalent to $(A4)$), so that $\hB_n^+ / \hN_n^+$ is
a homomorphic image of $S_n$, and in particular $\# (\hB_n^+ / \hN_n^+) \le n!$.
All in all, $\# \hB_n^+ \le 2^{n-1}n!$ as required.

\end{proof}

\bigskip

The above presentation may be extended to the whole group
$B_n=G(2,n)$. % and to other wreath products $G(r,n)$. %=\ZZ_r \wr S_n$.

\begin{proposition}\label{t.relation1}
The hyperoctahedral group $B_n=G(2,n)$ is isomorphic to the abstract group
generated by $S_{2,n}:=\{a_i\,:\, 1 \le i \le n-1\}\cup \{b_i\,:\, 1 \le i \le n\}$
with defining relations
\[
b_i^2=1 \qquad (1 \le i \le n), \leqno(B1)
\]
\[
b_ib_j=b_jb_i \qquad (1 \le i< j \le n), \leqno(B2)
\]
\[
a_i^2=b_i b_{i+1} \qquad (1 \le i \le n-1), \leqno(B3)
\]
\[
a_i a_j = a_j a_i \qquad (|i-j|>1), \leqno(B4)
\]
\[
a_i a_{i+1} a_i= a_{i+1} a_i a_{i+1} \qquad (1 \le i \le n-1), \leqno(B5)
\]
\[
a_i b_j =b_j a_i \qquad (j \ne i, i+1), \leqno(B6)
\]
\[
a_i b_i= b_{i+1} a_i \qquad(1 \le i \le n-1) \leqno(B7)
\]
and
\[
a_i b_{i+1}= b_i a_i \qquad(1 \le i \le n-1). \leqno(B8)
\]

\end{proposition}

%proof via multiplicity free words (namely, words with no two
%identical consequent letters).

\smallskip

\begin{rem}
Note that relations $(A1)-(A4)$ in Proposition~\ref{t.relation2}
follow from relations $(B1)-(B8)$ in Proposition~\ref{t.relation1}.
Relation $(A1)$ follows from relations $(B1)$, $(B2)$ and $(B3)$.
Relations $(A2)-(A3)$ are relations $(B4)-(B5)$.
Finally, relation $(A4)$ follows from relations $(B1)$, $(B3)$, $(B5)-(B8)$
as follows:
\begin{eqnarray*}
(a_i a_{i+1})^3
& = & (a_i a_{i+1} a_i)(a_{i+1} a_i a_{i+1})
\ \,\, = \ \ (a_i a_{i+1} a_i)(a_i a_{i+1} a_i) \\
& = & a_i a_{i+1} b_i b_{i+1} a_{i+1} a_i
\ \ \ \ \ \ \ \, = \ \ b_{i+1} a_i a_{i+1}  a_{i+1} a_i b_{i+2}\\
& = & b_{i+1} a_i b_{i+1} b_{i+2} a_i b_{i+2}
\ \ \ \ \ = \ \ b_{i+1} b_i a_i a_i b_{i+2} b_{i+2}\\
& = & b_{i+1} b_i b_i b_{i+1} b_{i+2} b_{i+2}
\ \ \ \ \ \, =\ \ 1.
\end{eqnarray*}
%First equality follows from $(B5)$; second, fourth and sixth
%equalities follow from $(B3)$; third and fifth from $(B6)$, $(B7)$
%and $(B8)$; and last one from $(B1)$.
\end{rem}

\begin{proof}
Similar to the proof of Proposition~\ref{t.relation2} (and somewhat simpler).

Let $\hB_n$ be the abstract group with the presentation described in Proposition~\ref{t.relation1}.
Define a map $\phi$ from the free group generated by $a_1, \ldots, a_{n-1}, b_1, \ldots, b_n$
to $B_n$ by
\[
\phi(a_i):= [1, \ldots, -(i+1),  i, \ldots,n] \qquad (1 \le i \le
n-1).
\]
and
$$
\phi(b_i):=[1,\ldots, - i,\ldots, n] \qquad (1 \le i \le n).
$$
%We claim that $\phi$, when extended multiplicatively, determines
%a well-defined group isomorphism  $\phi: \hB_n^+ \to B_n^+$.
Thus $\phi(a_i) = (i,i+1)(i,-i)$ and $\phi(b_i) = (i,-i)$. It is
easy to check that relations $(B1)-(B8)$ are satisfied when each
$a_i$ ($b_i$) is replaced by $\phi(a_i)$ ($\phi(b_i)$,
respectively). This therefore defines a group homomorphism, which
we again denote by $\phi$, from $\hB_n$ to $B_n$. We shall show
that it is actually an isomorphism.

Clearly, $\{\phi(a_i)\phi(b_i)\,:\, 1 \le i \le n-1\}$ is the set
of Coxeter generators for the symmetric group $S_n$, embedded
naturally into $B_n$. Similarly for $\{\phi(b_i)\,:\, 1 \le i \le
n\}$ and $\ZZ_2^n$. Since $B_n = \ZZ_2^n \rtimes S_n$, it follows
that $\{\phi(a_i)\,:\, 1 \le i \le n-1\}\cup \{\phi(b_i)\,:\, 1
\le i \le n\})$ generates $B_n$. Thus $\phi$ is surjective and, in
particular,
\[
\# \hB_n \ge \# B_n.
\]

It remains to show that $\phi$ is injective. Since it is surjective and
$\# B_n = 2^{n} n!$, it suffices to show that $\# \hB_n \le 2^{n} n!$.

Let $\hN_n$ be the subgroup of $\hB_n$ generated by $b_1, \ldots, b_{n}$.
We shall show that $\hN_n$ is a commutative normal subgroup of $\hB_n$.
Commutativity follows from $(B2)$,
while normality follows from $(B6) - (B8)$ which may be written as
\[
a_i^{-1} b_j a_i =b_j \qquad (j \ne i, i+1),
\]
\[
a_i^{-1} b_{i+1} a_i = b_i \qquad(1 \le i \le n-1)
\]
and
\[
a_i^{-1} b_i a_i = b_{i+1} \qquad(1 \le i \le n-1).
\]

We can now wrap up the proof:
$\hN_n$ is a commutative group generated by the involutions $b_1, \ldots, b_{n}$.
Thus each element of $\hN_n$ can be written as a product $b_{i_1} \cdots b_{i_k}$
for some $k \ge 0$ and $1\le i_1 < \ldots < i_k \le n$. In particular, $\# \hN_n \le 2^{n}$.
Also, $\hN_n$ is a normal subgroup of $\hB_n$.
The quotient $\hB_n / \hN_n$ is generated by ${\bar a}_i$, the cosets corresponding to
the generators $a_i$ of $\hB_n$ $(1 \le i \le n-1)$.
The ${\bar a}_i$ satisfy relations $(B4) - (B5)$, with $(B3)$ replaced by
\[
{\bar a}_i^2 = 1 \qquad (1 \le i \le n-1).
\]
These are exactly the Coxeter relations defining the symmetric group $S_n$,
so that $\hB_n / \hN_n$ is a homomorphic image of $S_n$,
and in particular $\# (\hB_n / \hN_n) \le n!$.
All in all, $\# \hB_n \le 2^{n}n!$ as required.

\end{proof}

\section{Flag Weak Order} %{Basic Concepts}

From now on we identify the abstract generating set of $B_n$
$$
S_{2,n}=\{a_i:\ 1\le
i< n\}\cup \{b_i:\ 1\le i\le n\}$$ with the choice %the embedding
$$a_i:=[1,\dots, i-1,-(i+1),i,i+2,\dots,n]$$ and
$$b_i:=[1,\dots,i-1,-i,i+1,\dots,n],$$ used in the proof of
Proposition~\ref{t.relation1}.

\bigskip

Following  Foata and Han~\cite{FH1, FH}, let the {\em flag inversion number} of %an $r$-colored permutation
$\pi\in B_n$ be
$$
\finv(\pi):=2\cdot\inver (|\pi|)+ \nega(\pi).
$$

\begin{defn}\label{weak-definition-B_n}
The  {\em flag (right) weak order} $\preceq$ on $B_n$ is the
reflexive and transitive closure of the relation
%$\pi\hat\preceq\ \sigma$ if $\sigma=\pi s$ for some $s\in S_{2,n}$ and $\finv(\pi)<\finv(\sigma)$.
$$\pi  \lessdot %\hat\preceq\
\pi s\ \ \ \ \  \Longleftrightarrow\ \ \ \ \  \pi\in G(2,n), s \in
S_{2,n}\ \ \ \text{ and }\ \ \ \finv(\pi)<\finv(\pi s).$$
\end{defn}

%Consider the conjugacy closure of $S_{2,n}$
%$$
%T_{2,n}:= \{(-i,i):\ 1\le i\le n\}\cup \{(i,-j,-i,j):\ 1\le i<
%j\le n\} \cup \{(i,j,-i,-j):\ 1\le i< j\le n\}.
%$$

%Define the {\it $F$-strong order} on $B_n$, $\le$, as the
%reflexive and transitive closure of the relation $\pi\le_S \sigma$
%if $\sigma=\pi t$ for some $t\in T_{r,n}$ and
%$\finv(\pi)<\finv(\sigma)$.

Note that this order is not isomorphic to the classical weak order on $B_n$.

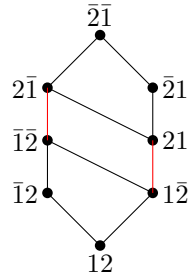
\begin{figure}[htb]
\begin{center}

\begin{tikzpicture}[scale=0.7]
\fill (1,4) circle (0.1) node[above]{$\bar 2 \bar 1$};

\fill (2,3) circle (0.1) node[right]{$\bar 2 1$};

\fill (0,3) circle (0.1) node[left]{$2\bar 1$};

\fill (2,2) circle (0.1) node[right]{$21$};

\fill (0,2) circle (0.1) node[left]{$\bar 1 \bar 2$};

\fill (2,1) circle (0.1) node[right]{$1 \bar 2$};

\fill (0,1) circle (0.1) node[left]{$\bar 1 2$};

\fill (1,0) circle (0.1) node[below]{$12$};

\draw (2,1)--(1,0)--(0,1)--(0,2)--(2,1);

\draw (2,2)--(2,3)--(1,4)--(0,3)--(2,2);

\draw[red] (2,1)--(2,2);

\draw[red] (0,2)--(0,3);

\end{tikzpicture}
\caption{\label{fig:B2} The Hasse diagram of the flag weak order
on $B_2$. Edges are colored black for $b_i$-s and red for $a_i$-s.}
\end{center}
\end{figure}

\begin{figure}[htb]
\begin{center}

\begin{tikzpicture}[scale=0.7]

\fill (4,13.5) circle (0.1) node[above]{\small $\bar 3\bar 2 \bar 1$};

\fill (6,12) circle (0.1) node[right]{\small $\bar 3 \bar 2 1$};

\fill (4,12) circle (0.1) node[left]{\small $\bar 3 2 \bar 1$};

\fill (2,12) circle (0.1) node[left]{\small $3 \bar 2 \bar 1$};

\fill (6,10.5) circle (0.1) node[right]{\small $\bar 3 2 1$};

\fill (8,10.5) circle (0.1) node[right]{\small $\bar 3 \bar 1\bar 2$};

\fill (4,10.5) circle (0.1) node[right]{\small $3 \bar 2 1$};

\fill (0,10.5) circle (0.1) node[left]{\small $\bar 2 \bar 3 \bar 1$};

\fill (2,10.5) circle (0.1) node[left]{\small $ 3 2 \bar 1$};

\fill (10,9) circle (0.1) node[right]{\small $\bar 3 \bar 1 2$};

\fill (8,9) circle (0.1) node[right]{\small $\bar 3 1 \bar 2$}; %*****

\fill (6,9) circle (0.1) node[right]{\small $3 \bar 1 \bar 2$};

\fill (4,9) circle (0.1) node[right]{\small $3 2 1$};

\fill (-2,9) circle (0.1) node[left]{\small $ 2 \bar 3 \bar 1  $};

\fill (0,9) circle (0.1) node[left]{\small $ \bar 2  3  \bar 1$};

\fill (2,9) circle (0.1) node[left]{\small $\bar 2 \bar 3  1$};

\fill (12,7.5) circle (0.1) node[right]{\small $\bar 1 \bar 3 \bar 2$};

\fill (10,7.5) circle (0.1) node[right]{\small $\bar 3 1 2$};

\fill (8,7.5) circle (0.1) node[right]{\small $3 \bar 1 2$};

\fill (6,7.5) circle (0.1) node[right]{\small $3 1 \bar 2$};

\fill (2,7.5) circle (0.1) node[left]{\small $\bar 2 3 1$};

\fill (0,7.5) circle (0.1) node[left]{\small $2 \bar 3 1$};

\fill (-2,7.5) circle (0.1) node[left]{\small $2 3 \bar 1$};

\fill (4,7.5) circle (0.1) node[left]{\small $\bar 2 \bar 1 \bar 3$};

\fill (14,6) circle (0.1) node[right]{\small $1 \bar 3 \bar 2$};

\fill (12,6) circle (0.1) node[right]{\small $\bar 1 3 \bar 2$};

\fill (10,6) circle (0.1) node[right]{\small $\bar 1 \bar 3 2$};

\fill (8,6) circle (0.1) node[right]{\small $3 1 2$};

\fill (6,6) circle (0.1) node[left]{\small $2 \bar 1 \bar 3$};

\fill (4,6) circle (0.1) node[left]{\small $\bar 2 1 \bar 3$};

\fill (2,6) circle (0.1) node[left]{\small $\bar 2 \bar 1 3$};

\fill (0,6) circle (0.1) node[left]{\small $2 3 1$};

\fill (14,4.5) circle (0.1) node[right]{\small $1 3 \bar 2$};

\fill (12,4.5) circle (0.1) node[right]{\small $1 \bar 3 2$};

\fill (10,4.5) circle (0.1) node[right]{\small $\bar 1 3 2$};

\fill (8,4.5) circle (0.1) node[left]{\small $\bar 1 \bar 2 \bar 3$};

\fill (6,4.5) circle (0.1) node[left]{\small $2 1 \bar 3$};

\fill (4,4.5) circle (0.1) node[left]{\small $2 \bar 1 3$};

\fill (2,4.5) circle (0.1) node[left]{\small $\bar 2 1 3$};

\fill (12,3) circle (0.1) node[right]{\small $1 3 2$};

\fill (10,3) circle (0.1) node[right]{\small $1 \bar 2 \bar 3$};

\fill (8,3) circle (0.1) node[left]{\small $\bar 1 2 \bar 3$};

\fill (6,3) circle (0.1) node[left]{\small $\bar 1 \bar 2 3$};

\fill (4,3) circle (0.1) node[left]{\small $2 1 3$};

\fill (10,1.5) circle (0.1) node[right]{\small $1 2 \bar 3$};

\fill (8,1.5) circle (0.1) node[right]{\small $1 \bar 2 3$};

\fill (6,1.5) circle (0.1) node[left]{\small $\bar 1 2 3$};

\fill (8,0) circle (0.1) node[below]{\small $123$};

\draw[red] (8,1.5)--(4,3);

\draw[red] (10,1.5)--(12,3);

\draw[red] (10,3)--(14,4.5);

\draw[red] (2,10.5)--(-2,9);

\draw[red] (10,3)--(6,4.5);

\draw[red] (6,3)--(4,4.5);  %--(2,4.5)

\draw[red] (8,3)--(10,4.5);

\draw[red] (4,10.5)--(2,9);

\draw[red] (6,4.5)--(0,6);

\draw[red] (6,9)--(12,7.5);

\draw[red] (8,4.5)--(12,6);

\draw[red] (8,4.5)--(6,6);

\draw[red] (4,7.5)--(0,9);

\draw[red] (8,6)--(12,4.5);

\draw[red] (10,6)--(8,7.5);

\draw[red] (4,6)--(2,7.5);

\draw[red] (6,6)--(-2,7.5);

\draw[red] (6,7.5)--(14,6);

%\draw[red] (6,10.5)--(0,9);

\draw[red] (6,10.5)--(8,9);     %--(10,9)

%\draw[red] (10,9)--(6,10.5);

\draw[red] (4,9)--(0,7.5);

\draw[red] (6,7.5)--(4,9);

\draw[red] (6,9)--(2,10.5);

\draw[red] (4,12)--(8,10.5);

\draw[red] (2,12)--(0,10.5);

\draw (2,12)--(4,13.5)--(4,12)--(2,10.5)--(2,12);

\draw (4,13.5)--(6,12)--(6,10.5)--(4,12);

\draw (2,12)--(4,10.5)--(6,12);

\draw (2,10.5)--(4,9)--(6,10.5);

\draw (4,10.5)--(4,9);

\draw (4,3)--(2,4.5)--(2,6)--(4,7.5)--(6,6)--(6,4.5)--(4,3);

\draw (4,3)--(4,4.5);

\draw (2,4.5)--(4,6)--(6,4.5);

\draw (2,6)--(4,4.5)--(6,6);

\draw (4,6)--(4,7.5);

\draw (12,3)--(10,4.5)--(10,6)--(12,7.5)--(14,6)--(14,4.5)--(12,3);

\draw (12,3)--(12,4.5);

\draw (10,4.5)--(12,6)--(14,4.5);

\draw (10,6)--(12,4.5)--(14,6);

\draw (12,6)--(12,7.5);

\draw (0,6)--(-2,7.5)--(-2,9)--(0,10.5)--(2,9)--(2,7.5)--(0,6);

\draw (0,6)--(0,7.5);

\draw (-2,7.5)--(0,9)--(2,7.5);

\draw (-2,9)--(0,7.5)--(2,9);

\draw (0,9)--(0,10.5);

\draw (8,6)--(6,7.5)--(6,9)--(8,10.5)--(10,9)--(10,7.5)--(8,6);

\draw (8,6)--(8,7.5);

\draw (6,7.5)--(8,9)--(10,7.5);

\draw (6,9)--(8,7.5)--(10,9);

\draw (8,9)--(8,10.5);

%\draw (12,3)--(10,4.5)--(13,6)--(12,4.5)--(12,3);

\draw (6,1.5)--(8,0)--(8,1.5)--(6,3)--(6,1.5);

\draw (8,0)--(10,1.5)--(10,3)--(8,1.5);

\draw (6,1.5)--(8,3)--(10,1.5);

\draw (6,3)--(8,4.5)--(10,3);

\draw (8,3)--(8,4.5);

%\draw[red] (2,1)--(2,2);

%\draw[red] (0,2)--(0,3);

\end{tikzpicture}
\caption{\label{fig:B3} The Hasse diagram of the flag weak order
on $B_3$. Edges are colored black for $b_i$-s and red for $a_i$-s.}
\end{center}
\end{figure}
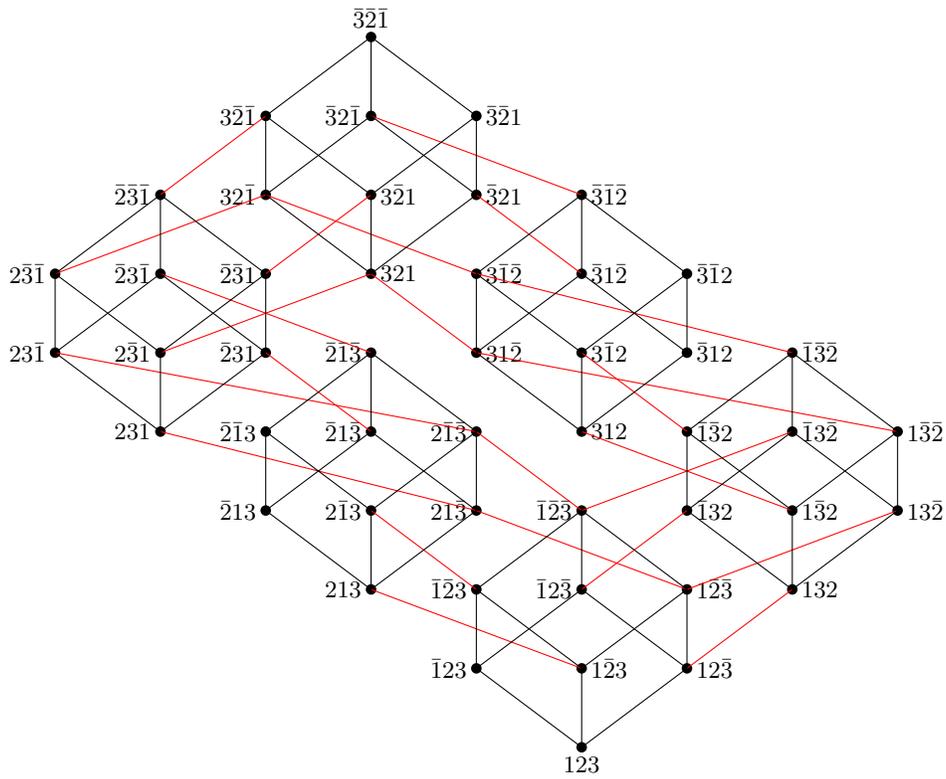

%%$S_{2,n}$ and $T_{2,n}$ should play an analogous role to the sets
%%of simple and all reflections respectively. For example,
%For $\pi\in \ZZ_2\wr S_n$ let
%$$ D(\pi):= \{t\in T_{2,n}:\
%\finv(\pi t)<\finv(\pi)\}.
%$$%, the following holds.

%\begin{claim}
%For every $\pi\in \ZZ_2\wr S_n$
%$$
%\finv(\pi)=%\# D(\pi)=
%\#\{t\in T_{2,n}:\ \finv(\pi t)<\finv(\pi)\}.
%$$
%\end{claim}

%\medskip

%By definition,

\begin{proposition}\label{t.properties}
The poset $(B_n,\preceq)$ is
\begin{itemize}
\item[(i)] ranked (by flag inversion number);

\item[(ii)] self-dual (by $\pi\mapsto \pi \mu_0$, where
$\mu_0:=[\bar n,\dots, \bar 1]$ is the unique maximal element in
this order);

\item[(iii)] rank-symmetric and unimodal. \end{itemize}
\end{proposition}

%The rank is given by the flag inversion number. Denote the maximal
%element by $\mu_0:=[\bar n,\dots, \bar 1]$.

\begin{proof}\
(i) In order to show that all maximal chains between two elements
have the same length (the difference between their $\finv$ values),
it suffices to show that if $\sigma=\pi s$, with $s\in S_{2,n}$ and $\finv(\pi)<\finv(\sigma)$,
then there exists $\pi\prec w \preceq \sigma$ with $\finv(w)=\finv(\pi)+1$.
If $s\in\{b_1,\dots,b_n\}$ then
\[
\finv(\pi s)-\finv(\pi) =
2 \cdot (\inv(|\pi s|)-\inv(|\pi|)) + (\nega(\pi s)-\nega(\pi)) =
2 \cdot 0 \pm 1=1
\]
(positive by the assumption $\finv(\pi)<\finv(\pi s)$).
In this case we can take $w:=\pi s=\sigma$.
Otherwise, $s \in \{a_1,\dots,a_{n-1}\}$. Then
\[
\finv(\pi s)-\finv(\pi) =
2 \cdot (\inv(|\pi s|)-\inv(|\pi|)) + (\nega(\pi s)-\nega(\pi)) =
2 \cdot (\pm 1) \pm 1.
\]
Being positive by assumption, this number is either $1$ or
$3$. In the first case, which occurs if $s=a_i$ $(1 \le i \le n)$,
$|\pi(i)|<|\pi(i+1)|$ and $\pi(i+1)<0$, we can again take $w:=\pi
s$. In the second case, which occurs if $s=a_i$,
$|\pi(i)|<|\pi(i+1)|$ and $\pi(i+1)>0$, we have, by relations (B1)
and (B7) in Proposition~\ref{t.relation1}, $\sigma=\pi b_{i+1} a_i
b_i$, with $\finv(\pi)+3=\finv(\pi b_{i+1})+2=\finv(\pi
b_{i+1}a_i)+1= \finv(\sigma)$, and we can take $w:=\pi b_{i+1}$.

\smallskip

\noindent (ii) Let $\mu_0:=[\bar n,\dots, \bar 1]$. Then, for
every $\pi\in B_n$,

\begin{eqnarray*}%\label{1}
\finv(\pi \mu_0) & = & 2\cdot
\inv\left(\left|\pi \mu_0\right|\right) +\nega (\pi \mu_0)\\
& = & 2\left[{n\choose
2}-\inv\left(\left|\pi\right|\right)\right]+\left[n-\nega(\pi)\right]\\
& = & \finv(\mu_0)-\finv(\pi).
\end{eqnarray*}

%Also, by (\ref{1}),
If $\sigma=\pi s$ with $s\in S_{2,n}$ and
$\finv(\pi)<\finv(\sigma)$ then $$\finv(\pi \mu_0)-\finv(\sigma
\mu_0)=\finv(\sigma)-\finv(\pi)>0.$$

%If $\sigma=\pi s$ with $s\in S_{2,n}$ and
%$\finv(\pi)<\finv(\sigma)$ then

Also, $\pi \mu_0= \sigma \mu_0 \tilde s$, where
$$
\tilde s= \mu_0^{-1} s^{-1} \mu_0 =
\begin{cases}
   b_{n+1-i},   & \text{ if } s=b_i;  \\
   a_{n-i}, & \text{ if } s=a_i.
\end{cases}
$$
It follows, by Definition~\ref{weak-definition-B_n}, that
$ \pi \preceq \sigma \Longleftrightarrow \sigma \mu_0\preceq \pi \mu_0$,
and since right multiplication by $\mu_0$ is a bijection on $B_n$,
this proves self-duality.

\smallskip

\noindent (iii) Rank-symmetry follows from (ii) (and (i)).
Unimodality follows from (i) together with
Proposition~\ref{finv-gf}.

\end{proof}

The proof of Proposition~\ref{t.properties} implies the following statement.

\begin{corollary}\label{t.cover}
$\sigma$ covers $\pi$ in $(B_n,\preceq)$ if and only if either
\begin{itemize}
\item[(i)]
there exists $1\le i\le n$ such that
$$
i \not\in \Neg(\pi) \ \ \ \ \ \ {\rm{and}}\ \ \ \ \ \  \sigma=\pi b_i;
$$
or
\item[(ii)]
there exists $1 \le i \le n-1$ such that
$$
i+1\in \Neg(\pi), \ \ \ \ \ |\pi(i)|<|\pi(i+1)| \ \ \ \ \ \  {\rm{and}}\ \ \ \ \ \ \sigma=\pi a_i.
$$
\end{itemize}
\end{corollary}

\medskip

\section{Properties of the Flag Weak Order}\label{s.weak}

%\subsection{Definition and Basic Properties}\ \\
\subsection{Lattice Structure}\ \\

For a set of pairs $A\subseteq \{(i,j):\ 1\le i< j\le n\}$ let $\M(A):= \{j:\ (i,j)\in A\}$.
For example, $\M(\{ (1,6), (1,4), (2,3), (4,6)\})= \{ 3,4,6\}$.

\begin{proposition}\label{t.criterion}
For every $\pi, \sigma\in B_n$,
\begin{eqnarray}\label{e.preceq}
\pi \preceq \sigma &\Longleftrightarrow&
\Inv(|\pi^{-1}|) \subseteq \Inv(|\sigma^{-1}|) {\text{\rm \ \ and \ }}\\
\nonumber & & \Neg(\pi^{-1})\setminus \Neg(\sigma^{-1})\subseteq
\M\left[\Inv(|\sigma^{-1}|) \setminus \Inv(|\pi^{-1}|)\right].
\end{eqnarray}
\end{proposition}

\begin{proof}\ \\
\noindent
$ \Longrightarrow$ :
It suffices to show that the RHS of (\ref{e.preceq}) holds
whenever $\sigma$ covers $\pi$ in $(B_n, \preceq)$.
By Corollary~\ref{t.cover}, there are two cases to check:
\begin{itemize}
\item[(i)]
There exists $1\le i\le n$ such that $i\not\in\Neg(\pi)$ and $\sigma=\pi b_i$.
Then clearly $\Inv(|\pi^{-1}|)=\Inv(|\sigma^{-1}|)$ and
$\Neg(\pi^{-1})\setminus\Neg(\sigma^{-1})=\emptyset$.
\item[(ii)]
There exists $1 \le i \le n-1$ such that
$i+1\in \Neg(\pi)$, $|\pi(i)|<|\pi(i+1)|$ and $\sigma=\pi a_i$.
Denoting $p:=|\pi(i)|$ and $q:=|\pi(i+1)|$ we have
\[
p<q,\qquad
\sigma(i+1) = \pi(i) = \pm p,\qquad
\sigma(i) = -\pi(i+1) = q.
\]
Thus $\Inv(|\sigma^{-1}|) = \Inv(|\pi^{-1}|)\cup \{(p,q)\}$ and
$\Neg(\pi^{-1}) \setminus \Neg(\sigma^{-1}) = \{q\}$.
\end{itemize}

\noindent
$ \Longleftarrow$ :
Assume that the RHS of (\ref{e.preceq}) holds. There are two cases:
\begin{itemize}
\item[(i)]
$\Inv(|\pi^{-1}|)= \Inv(|\sigma^{-1}|)$.
Then $|\pi^{-1}|= |\sigma^{-1}|$ and %also
$\Neg(\pi^{-1})\setminus\Neg(\sigma^{-1})=\emptyset$, i.e.,
$\Neg(\pi^{-1})\subseteq \Neg(\sigma^{-1})$.
It is clear that one can get from $\pi$ to $\sigma$ by
a sequence of right multiplications by various $b_i$,
each step increasing $\finv(\cdot)$ by $1$.
Thus $\pi\preceq \sigma$.

\item[(ii)]
$\Inv(|\pi^{-1}|)$ is strictly contained in $\Inv(|\sigma^{-1}|)$. %\\
%{\bf INSERTED PROOF:}
%By (\ref{e.inv}),
Thus $|\pi|$ is strictly smaller than $|\sigma|$ in the right weak order on $S_n$,
and one can get from $|\pi|$ to $|\sigma|$ by
a sequence of right multiplications by various Coxeter generators $s_i$ of $S_n$,
each step increasing the cardinality of the inversion set by $1$.
Let $s_{i_1}, \ldots, s_{i_k}$ be such a sequence,
so that $|\sigma|=|\pi| s_{i_1} \cdots  s_{i_k}$.
%and $k= \inv(|\sigma|)- \inv(|\pi|)$.
Let $a_{i_1}, \ldots, a_{i_k}$ be the corresponding sequence of generators of $B_n$.
Define $\pi_0 := \pi$ and, recursively,
\[
\pi_j := \pi_{j-1} \tilde{a}_{j} \qquad(1 \le j \le k),
\]
where
\[
\tilde{a}_{j} := \begin{cases}
a_{i_j}, &\mbox{if } i_j+1 \in \Neg(\pi_{j-1});\\
b_{i_j+1} a_{i_j}, & \mbox{otherwise}.
\end{cases}
\]
It is easy to see that
\[
\pi = \pi_0 \preceq \pi_1 \preceq \ldots \preceq \pi_k,
\]
with $\finv(\pi_j) - \finv(\pi_{j-1}) \in \{1, 2\}$ $(\forall j)$.
We shall show that $\pi_k \preceq \sigma$, implying $\pi \preceq \sigma$.

Indeed $|\pi_k| = |\sigma|$, and in particular $\Inv(|\pi_k^{-1}|) = \Inv(|\sigma^{-1}|)$.
Also, for each $1 \le j \le k$,
\[
\Neg(\pi_j^{-1}) = \begin{cases}
\Neg(\pi_{j-1}^{-1}) \setminus \{|\pi_{j-1}(i_j+1)|\}, &\mbox{if } i_j+1 \in \Neg(\pi_{j-1});\\
\Neg(\pi_{j-1}^{-1}), & \mbox{otherwise}.
\end{cases}
\]
Since $\M\left[\Inv(|\pi_j^{-1}|) \setminus \Inv(|\pi_{j-1}^{-1}|)\right] = \{|\pi_{j-1}(i_j+1)|\}$
we conclude that, in both cases,
\[
\Neg(\pi_{j}^{-1}) = \Neg(\pi_{j-1}^{-1}) \setminus \M\left[\Inv(|\pi_j^{-1}|) \setminus \Inv(|\pi_{j-1}^{-1}|)\right].
\]
Thus
\begin{eqnarray*}
\Neg(\pi_{k}^{-1})
&=& \Neg(\pi_{0}^{-1}) \setminus
\bigcup_{j=1}^{k} \M\left[\Inv(|\pi_j^{-1}|) \setminus \Inv(|\pi_{j-1}^{-1}|)\right]\\
&=& \Neg(\pi_0^{-1}) \setminus
\M\left[\Inv(|\pi_k^{-1}|) \setminus \Inv(|\pi_{0}^{-1}|)\right],
\end{eqnarray*}
where we have used the property $\bigcup_j M[A_j] = M[\bigcup_j A_j]$ and
the fact that $\Inv(|\pi_{j-1}^{-1}|) \subseteq \Inv(|\pi_{j}^{-1}|)$.
Since $\pi_0 = \pi$ and $\Inv(|\pi_k^{-1}|) = \Inv(|\sigma^{-1}|)$, we conclude that
\[
\Neg(\pi_k^{-1}) = \Neg(\pi^{-1}) \setminus M\left[\Inv(|\sigma^{-1}|) \setminus \Inv(|\pi^{-1}|)\right].
\]
Our assumption
\[
\Neg(\pi^{-1})\setminus \Neg(\sigma^{-1}) \subseteq
\M\left[\Inv(|\sigma^{-1}|) \setminus \Inv(|\pi^{-1}|)\right]
\]
is equivalent to
\[
\Neg(\sigma^{-1}) \supseteq \Neg(\pi^{-1}) \setminus
\M\left[\Inv(|\sigma^{-1}|) \setminus \Inv(|\pi^{-1}|)\right] ,
\]
namely to
\[
\Neg(\sigma^{-1}) \supseteq \Neg(\pi_k^{-1}).
\]
Together with  $\Inv(|\pi_k^{-1}|) = \Inv(|\sigma^{-1}|)$ this implies, by case (i) above,
that $\pi_k \preceq \sigma$.
\end{itemize}

\end{proof}

\begin{proposition}\label{t.lattice}
The poset $(B_n,\preceq)$ is a lattice.
\end{proposition}

\begin{proof}
For simplicity of notation, let
\begin{equation}\label{M-notation}
\M(|\pi|,|\sigma|) := \M\left[\Inv(|\sigma^{-1}|) \setminus \Inv(|\pi^{-1}|)\right] \qquad(\pi,\sigma \in B_n).
\end{equation}
Proposition~\ref{t.criterion} can be stated as:
\[
\pi \preceq \sigma \ \ \Longleftrightarrow \ \
\Inv(|\pi^{-1}|) \subseteq \Inv(|\sigma^{-1}|) {\text{\rm \ \ and \ }}
\Neg(\pi^{-1})\subseteq \Neg(\sigma^{-1}) \cup \M(|\pi|, |\sigma|).
\]
Let  $\sigma_1, \sigma_2$ be two elements of $B_n$. It follows that, for any $\pi \in B_n$,
\begin{eqnarray}\label{e.preceq2}
& & \pi \preceq \sigma_1 \text{\rm \ \ and \ } \pi \preceq \sigma_2 \Longleftrightarrow\\
\nonumber & & \qquad\qquad \Inv(|\pi^{-1}|)\subseteq \Inv(|\sigma_1^{-1}|) \cap \Inv(|\sigma_2^{-1}|) {\text{\rm \ \ and \ }}\\
\nonumber & & \qquad\qquad \Neg(\pi^{-1}) \subseteq
\left(\Neg(\sigma_1^{-1}) \cup \M(|\pi|,|\sigma_1|)\right) \cap
\left(\Neg(\sigma_2^{-1}) \cup \M(|\pi|,|\sigma_2|)\right).
\end{eqnarray}

We shall now define a candidate for the meet (in $B_n$) of $\sigma_1$ and $\sigma_2$,
and prove that it has the required properties.
First note that the intersection of inversion sets (of permutations in $S_n$)
is not necessarily an inversion set. Nevertheless, since $S_n$ under right weak order is
a lattice, there exists a meet
\[
\tau = |\sigma_1| \wedge_{S_n} |\sigma_2| \in S_n
\]
which satisfies, by~(\ref{e.inv}),
\[
\Inv(\tau^{-1})\subseteq \Inv(|\sigma_1^{-1}|) \cap
\Inv(|\sigma_2^{-1}|)
\]
and
\begin{equation}\label{e.meet_inv}
\Inv(\gamma^{-1})\subseteq \Inv(|\sigma_1^{-1}|) \cap
\Inv(|\sigma_2^{-1}|) \ \ \Longrightarrow \ \
\Inv(\gamma^{-1})\subseteq \Inv(\tau^{-1}) \qquad(\forall \gamma
\in S_n).
\end{equation}
Define $\sigma_\wedge \in B_n$ by
\begin{equation}\label{e.meet_absolute}
|\sigma_\wedge| := |\sigma_1| \wedge_{S_n} |\sigma_2| \quad (=\tau)
\end{equation}
and
\begin{equation}\label{e.meet_Neg}
\Neg(\sigma_\wedge^{-1}) := \left(\Neg(\sigma_1^{-1}) \cup
\M(\tau,|\sigma_1|)\right) \cap \left(\Neg(\sigma_2^{-1}) \cup
\M(\tau,|\sigma_2|)\right).
\end{equation}
Then clearly $\sigma_\wedge \preceq \sigma_1$ and $\sigma_\wedge \preceq \sigma_2$.
It remains to show that $\pi \preceq \sigma_1$ and $\pi \preceq \sigma_2$ implies $\pi \preceq \sigma_\wedge$.
This is straightforward if $\Inv(|\pi^{-1}|) = \Inv(|\sigma_\wedge^{-1}|)$, but more intricate otherwise.

Assume that $\pi \preceq \sigma_1$ and $\pi \preceq \sigma_2$. Then
\[
\Inv(|\pi^{-1}|)\subseteq \Inv(|\sigma_1^{-1}|) \cap \Inv(|\sigma_2^{-1}|)
\]
so that, by~(\ref{e.meet_inv}),
\[
\Inv(|\pi^{-1}|)\subseteq \Inv(\tau^{-1}) =
\Inv(|\sigma_\wedge^{-1}|).
\]
From
\[
\Inv(|\pi^{-1}|)\subseteq \Inv(|\sigma_\wedge^{-1}|) \subseteq \Inv(|\sigma_1^{-1}|)
\]
it now follows that
\[
\Inv(|\sigma_1^{-1}|) \setminus \Inv(|\pi^{-1}|) =
\left(\Inv(|\sigma_1^{-1}|) \setminus \Inv(|\sigma_\wedge^{-1}|)\right) \cup
\left(\Inv(|\sigma_\wedge^{-1}|) \setminus \Inv(|\pi^{-1}|)\right)
\]
and therefore
\[
\M(|\pi|,|\sigma_1|) = \M(|\pi|,|\sigma_\wedge|) \cup \M(|\sigma_\wedge|,|\sigma_1|);
\]
similarly for $\sigma_2$. From~(\ref{e.preceq2}) it thus follows that
\begin{eqnarray*}
\Neg(\pi^{-1}) &\subseteq&
\left(\Neg(\sigma_1^{-1}) \cup \M(|\pi|,|\sigma_1|)\right) \cap
\left(\Neg(\sigma_2^{-1}) \cup \M(|\pi|,|\sigma_2|)\right) \\
&=&
\left(\Neg(\sigma_1^{-1}) \cup \M(|\pi|,|\sigma_\wedge|) \cup \M(|\sigma_\wedge|,|\sigma_1|)\right) \cap \\
& &
\left(\Neg(\sigma_2^{-1}) \cup \M(|\pi|,|\sigma_\wedge|) \cup \M(|\sigma_\wedge|,|\sigma_2|)\right) \\
&=&
\M(|\pi|,|\sigma_\wedge|) \cup \\
& & \left[\left(\Neg(\sigma_1^{-1}) \cup \M(|\sigma_\wedge|,|\sigma_1|)\right) \cap
\left(\Neg(\sigma_2^{-1}) \cup \M(|\sigma_\wedge|,|\sigma_2|)\right)\right] \\
&=&
\M(|\pi|,|\sigma_\wedge|) \cup \Neg(\sigma_\wedge^{-1}),
\end{eqnarray*}
using definition~(\ref{e.meet_Neg}) of $\Neg(\sigma_\wedge^{-1})$.
In other words, $\pi \preceq \sigma_\wedge$ as required.

We have shown the existence of meets in $(B_n,\preceq)$. The existence of joins follows by self-duality
(Proposition~\ref{t.properties}(ii)):
\[
\sigma_1 \vee \sigma_2 = (\sigma_1 \mu_0 \wedge \sigma_2 \mu_0) \mu_0.
\]
\end{proof}

Note that (\ref{e.meet_absolute})-(\ref{e.meet_Neg}) in the proof of
Proposition~\ref{t.lattice} provide an explicit description of the meet of two elements.
% in was given in the proof of Proposition~\ref{t.lattice}.
One can generalize this description to any number of elements, using the notation
$\M(|\pi|,|\sigma|)$ from~(\ref{M-notation}). For the corresponding description of the join
it is convenient to use also the notation $\Pos(\sigma) := [n] \setminus \Neg(\sigma)$ for $\sigma \in B_n$.

\begin{lemma}\label{t.meet_join}
Let $A$ be an arbitrary subset of $B_n$.
%, let $\tau$ be the meet of the corresponding unsigned permutations in the weak order on $S_n$,
%that is $%\left|\bigwedge_{\sigma\in A} \pi \right| \\tau:= \bigwedge_{\sigma\in A} |\sigma|.$
%Then
%For every subset $A\subseteq B_n$
\begin{itemize}
\item[(i)]
The meet $A_\wedge$ of $A$ in $(B_n,\preceq)$ is determined by
\[
|A_\wedge| := \bigwedge_{\sigma\in A} |\sigma|,
\]
where the meet is taken with respect to the (right) weak order on $S_n$, and by
\[
\Neg(A_\wedge^{-1}) := \bigcap_{\sigma\in A}
\left( \Neg(\sigma^{-1}) \cup \M(|A_\wedge|, |\sigma|) \right).
\]

\item[(ii)]
The join $A_\vee$ of $A$ in $(B_n,\preceq)$ is determined by
\[
|A_\vee| := \bigvee_{\sigma \in A} |\sigma|
\]
and by
\[
\Pos(A_\vee^{-1}) := \bigcap_{\sigma \in A} 
\left( \Pos(\sigma^{-1}) \cup \M(|\sigma|, |A_\vee|) \right).
\]

\end{itemize}

\end{lemma}

\medskip

\begin{rem}
For $n\ge 3$ $(B_n,\preceq)$ is not semi-modular. To verify that
notice that $\pi=2 \bar 1 \bar 3$ and $\sigma=\bar 1 3 \bar 2$
cover their meet $\pi\wedge \sigma=\bar 1 \bar 2 \bar 3$ but are
not covered by their join $\pi\vee \sigma=32 \bar 1$.
%${\rm rank}(2\bar 3 1 \vee 31 \bar 2)+{\rm rank}(2\bar 3 1 \wedge 31 \bar 2)=\finv (321)+\finv (1\bar 2 \bar 3)=6+2=8$
%while the sum of their ranks is $\finv((2\bar 3 1)+\finv(31 \bar 2)=10$.
Also, for $n\ge 2$ $(B_n,\preceq)$ is not complemented,
since $1\bar 2$ has no complement in $(B_2,\preceq)$.
\end{rem}

\smallskip

\subsection{Homotopy Type and M\"obius Function}\ \\

The following results generalize well-known properties of the
classical weak order on a Coxeter group.
Recall that an {\em atom} in an interval $[\pi, \sigma]$ is 
an element $\tau \in [\pi, \sigma]$ covering $\pi$.
Recall also the notation $A_\vee$ from Lemma~\ref{t.meet_join}(ii).

\begin{lemma}\label{t.injective}
Suppose that $\pi\prec \sigma$ in $B_n$. %and $\finv(\sigma)-\finv(\pi)\ge 2$. 
Then, for any two sets $A$ and $B$ of atoms in the interval $[\pi, \sigma]$,  
\[
A \ne B \,\Longrightarrow\, A_\vee \ne B_\vee.
\]
\end{lemma}

\begin{proof}
For a set $A$ of atoms  in the interval $[\pi, \sigma]$ denote
\[
A_1 := A \cap \{\pi a_i\,:\,1 \le i \le n-1\}
\]
and
\[
A_2 := A \cap \{\pi b_i\,:\,1 \le i \le n\}.
\]
Assume now that $A$ and $B$ are sets of atoms in $[\pi, \sigma]$ such that $A_\vee = B_\vee$.
We shall prove that $A = B$.

Since $|\pi b_i| = |\pi|$ for all $i$, it follows from Lemma~\ref{t.meet_join}(ii) that
\[
|A_\vee| = \bigvee_{\tau \in A} |\tau| = \bigvee_{\tau \in A_1} |\tau|,
\]
where joins are taken in $S_n$; and therefore
\[
A_\vee = B_\vee  \,\Longrightarrow\, |A_\vee| = |B_\vee| \,\Longrightarrow\, A_1 = B_1. 
\]
The latter implication holds since joins of sets of atoms uniquely determine the sets
in any interval in the usual weak order on $S_n$; see, e.g., \cite[Lemma 3.2.4(i)]{BB}.

We still need to show that $A_2 = B_2$. 
%Notice, first, that if 
If $\sigma = \pi a_i$ covers $\pi$ then, by definition,
$\Pos(\sigma^{-1})$ is the (disjoint) union of $\Pos(\pi^{-1})$ and $\{|\pi(i+1)|\}$, 
while $\M(|\pi|, |{A}_\vee|)$ is the (not necessarily disjoint) union 
of $\M(|\sigma|, |{A}_\vee|)$ and $\{|\pi(i+1)|\}$. 
Hence, for every $\sigma \in A_1$,
\[
\Pos(\sigma^{-1}) \cup \M(|\sigma|, |{A}_\vee|) = 
\Pos(\pi^{-1}) \cup \M(|\pi|, |{A}_\vee|).
\]
On the other hand, if $\sigma = \pi b_i \in A_2$ then
\[
\Pos(\sigma^{-1}) \cup \M(|\sigma|, |{A}_\vee|) = 
(\Pos(\pi^{-1}) \setminus \{\pi(i)\}) \cup \M(|\pi|, |{A}_\vee|).
\]
Thus, by Lemma~\ref{t.meet_join}(ii),
\[
\Pos(A_\vee^{-1}) = 
\bigcap_{\sigma \in A_1 \cup A_2} \left(
\Pos(\sigma^{-1}) \cup \M(|\sigma|, |A_\vee|) \right)
%\cap
%\bigcap_{\sigma \in A_2} \left( \Pos(\sigma^{-1}) \cup
%\M(|\sigma|, |A_\vee|) \right)
\]
\[
= \left( \Pos(\pi^{-1}) \cup \M(|\pi|, |{A}_\vee|) \right) \cap
\bigcap_{\sigma \in A_2} \left( \Pos(\sigma^{-1}) \cup
\M(|\sigma|, |A_\vee|) \right)
\]
\[
= \left( \Pos(\pi^{-1})\cup \M(|\pi|, |{A}_\vee|) \right) \cap
\bigcap_{\pi b_i \in A_2} \left( (\Pos(\pi^{-1})\setminus \{\pi(i)\})
\cup \M(|\pi|, |A_\vee|) \right)
\]
\[
= \left( \Pos(\pi^{-1})\setminus \{\pi(i) \,:\, \pi b_i \in A_2\} \right) 
\cup \M(|\pi|, |{A}_\vee|).
\]
(The intermediate steps, though not the end result, should be slightly rewritten if $A_1 = \emptyset$.)
If we show that 
\begin{equation}\label{e.pi_bi}
\pi b_i \in A_2 \,\then\, \pi(i) \not\in \M(|\pi|, |{A}_\vee|)
\end{equation}
it will then follow that, assuming $|A_\vee| = |B_\vee|$,
\begin{equation}\label{e.Pos_eq}
\Pos(A_\vee^{-1}) = \Pos(B_\vee^{-1}) \,\Longleftrightarrow\, A_2 = B_2.
\end{equation}
Indeed, 
\[
\pi b_i \in A_2 \,\then\, \pi(i) > 0 \,\then\, \pi a_{i-1} \not\in A_1.
\]
An examination of $|A_\vee|$ as a join of atoms in an interval of $S_n$ shows that
\[
\M(|\pi|, |{A}_\vee|) \subseteq \{|\pi(i)| \,:\, \pi a_{i-1} \in A_1\},
\]
which implies (\ref{e.pi_bi}) and (\ref{e.Pos_eq}) and completes the proof. 

\end{proof}

Lemma~\ref{t.injective} leads to an easy way to determine 
the homotopy type and M\"obius function for open intervals in $(B_n, \preceq)$,
generalizing~\cite[Theorem 3.2.7 and Corollary 3.2.8]{BB}.

\begin{proposition}\label{t.homotopy-type}
Suppose that $\pi\prec \sigma$ in $B_n$ and $\finv(\sigma)-\finv(\pi)\ge 2$.
Then the order complex of the open interval $(\pi,\sigma)$ is
homotopy equivalent to the sphere ${\bf S}^{k-2}$
if $\sigma$ is the join of $k$ atoms in the interval $[\pi, \mu_0]$,
and is contractible otherwise.
\end{proposition}

\begin{corollary}\label{t.mobius}
For every $\pi, \sigma\in B_n$,
\[
\mu(\pi,\sigma)=\begin{cases}
   (-1)^k,   & \text{ if } \sigma \text{ is the join of } k \text{ atoms in } [\pi, \mu_0];  \\
   0, & \text{ otherwise. }
\end{cases}
\]
\end{corollary}

\medskip

The proofs of Proposition~\ref{t.homotopy-type} 
and Corollary~\ref{t.mobius} are along the lines of 
the analogous proofs for the symmetric group~\cite[Theorem 3.2.7 and Corollary 3.2.8]{BB},
and are left to the reader.

\smallskip

\subsection{Tits Property}\ \\

In this subsection it will be shown that maximal chains in $(B_n, \preceq)$
exhibit a Tits-type connectivity property.

Let $\pi, \sigma \in B_n$ such that $\pi \preceq \sigma$.
Each maximal chain in the interval $[\pi, \sigma]$ of $(B_n, \preceq)$ corresponds to
a unique word $w = s_{i_1} \cdots s_{i_d}$ of length $d = \finv(\sigma) - \finv(\pi)$
with letters $s_{i_j}$ in the alphabet $S_{2,n}$, such that
$\finv(\pi s_{i_1} \cdots s_{i_j}) - \finv(\pi) = j$ for all $1 \le j \le d$.

\begin{proposition}\label{t.Tits} {(Tits Property)}
Any two maximal chains in any interval $[\pi, \sigma]$ of $(B_n, \preceq)$ are connected
via the following pseudo-Coxeter moves on the corresponding words:
$$
b_i b_j \longleftrightarrow b_j b_i \qquad (1\le i < j\le n), \leqno(T1)
$$
$$
a_i b_j \longleftrightarrow b_j a_i \qquad (j\ne i, i+1), \leqno(T2)
$$
$$
a_i b_{i+1} \longleftrightarrow b_i a_i \qquad(1 \le i \le n-1), \leqno(T3)
$$
$$
a_i a_j \longleftrightarrow a_j a_i \qquad (|i-j|>1), \leqno(T4)
$$
and
$$
a_i a_{i+1} b_{i+1} a_i \longleftrightarrow a_{i+1} b_{i+1} a_i a_{i+1} \qquad(1 \le i \le n-1). \leqno(T5)
$$
\end{proposition}

In order to prove that we first classify maximal chains in certain special intervals.
% with exactly two atoms.

%Recall that atomic intervals with exactly two atoms are intervals
%from $\pi\in B_n$ with $\pi\preceq \pi s$ and $\pi\preceq \pi s'$
%for some pair $s, s'\in S_{2,n}$, to $\pi s \vee \pi s'$.

\begin{lemma}\label{t.2-atoms}
Let $\pi\in B_n$, $s, s'\in S_{2,n}$, $s \ne s'$ such that both $\pi s$ and $\pi s'$ cover $\pi$.
%let $\alpha_\pi(s,s')$ be a labelling of a maximal chain from $\pi$ to $\pi s \vee \pi s'$ which starts with $s$.
Then:
\begin{itemize}
\item[(i)]
%For $\{s,s'\} = \{a_i,b_{i+1}\}$ this is impossible:
%There are no $\pi\in B_n$ and $1 \le i \le n-1$ such that
%both $\pi a_i$ and $\pi b_{i+1}$ cover $\pi$.
The interval $[\pi, \pi s \vee \pi s']$ contains exactly two maximal chains,
one described by a word starting with $s$ and one by a word starting with $s'$.

\item[(ii)]
%For $\{s,s'\}\ne \{a_i,b_{i+1}\}$ there exists a unique maximal chain from $\pi$ to $\pi s \vee \pi s'$
%which starts with $s$. It is independent of $\pi$, and will be denoted $\alpha(s,s')$.
%$\alpha_\pi(s,s')$ always exist, is unique and is independent of $\pi$.\\ It will be denoted $\alpha(s,s')$.
The above words are independent of $\pi$, as long as both $\pi s$ and $\pi s'$ cover $\pi$.
Denote by $\alpha(s,s')$ the word corresponding to the chain starting with $s$.

\item[(iii)]
The complete list of words corresponding to maximal chains in
intervals of the form $[\pi, \pi s \vee \pi s']$ in $(B_n, \preceq)$ is:
\[
\alpha(b_i,b_j)=b_i b_j \qquad (i \ne j);
\]
\[
\alpha (a_i, b_j) = a_i b_j, \quad %\qquad (j \ne i,i+1);
\alpha (b_j, a_i) =  b_j a_i \qquad (j \ne i,i+1);
\]
\[
\alpha (a_i, b_i)= a_i b_{i+1},\quad %\qquad (1 \le i \le n-1);
\alpha (b_i, a_i)= b_i a_i \qquad (1 \le i \le n-1);
\]
\[
\alpha(a_i, a_j)= a_i a_j \qquad (|i-j|>1);
\]
and
\[
\alpha(a_i, a_{i+1})= a_i a_{i+1} b_{i+1} a_i, \quad %\qquad (1 \le i \le n-2);
\alpha(a_{i+1}, a_i)= a_{i+1} b_{i+1} a_i a_{i+1} \qquad (1 \le i \le n-2).
\]
$\alpha (a_i, b_{i+1})$ and $\alpha (b_{i+1}, a_i)$ do not exist, since
$\pi a_i$ and $\pi b_{i+1}$ cannot cover $\pi$ simultaneously.
\end{itemize}

\end{lemma}

%\begin{proof}
\noindent
{\it Proof of Lemma~\ref{t.2-atoms}.}
First of all, $\pi a_i$ and $\pi b_{i+1}$ cannot cover $\pi$ simultaneously since,
by Corollary~\ref{t.cover},
\[
\pi a_i \text{\ covers } \pi
\,\Longrightarrow\, \pi(i+1) < 0
\,\Longrightarrow\, \pi b_{i+1}\prec \pi.
\]
%\smallskip

We shall deal with the other cases one by one.

\smallskip

If $i \ne j$ then, by Corollary~\ref{t.cover},
$\pi b_i$ and $\pi b_j$ cover $\pi$ if and only if $\pi(i) > 0$ and $\pi(j) > 0$. %$i,j\not\in \Neg(\pi)$.
Then $\pi b_i\vee \pi b_j = \pi b_i b_j$, and indeed
$\alpha(b_i, b_j) = b_i b_j$ is unique and independent of $\pi$.

\smallskip

%Similarly, by Corollary~\ref{t.cover}, for
If $j \ne i,i+1$ then, by Corollary~\ref{t.cover},
%$\pi\preceq \pi b_j$ if and only if $j\not\in \Neg(\pi)$ if and only if $\pi a_i
%\preceq \pi a_i b_j$. By similar arguments, if $j\ne i,i+1$, then
%$\pi \preceq \pi a_i \Longleftrightarrow \pi b_j \preceq \pi b_j
%a_i $, implying $\pi a_i\vee \pi b_j= \pi a_i b_j=\pi b_j a_i$ and
%thus $\alpha (a_i, b_j)= a_i, b_j$ and $\alpha (b_j, a_i)= b_j, a_i$.
$\pi a_i$ and $\pi b_j$ cover $\pi$ if and only if $\pi(i+1) < 0$, $|\pi(i)| < |\pi(i+1)|$ and $\pi(j) > 0$.
Then $\pi a_i\vee \pi b_j = \pi a_i b_j = \pi b_j a_i$, so that
$\alpha(a_i, b_j) = a_i b_j$ and $\alpha(b_j, a_i) = b_j a_i$ are clearly unique and independent of $\pi$.

\smallskip

%If $\pi \preceq \pi a_i$ and $\pi \preceq \pi b_i$ for some $1\le
%i< n$ then by Corollary~\ref{t.cover}, $\pi \preceq \pi b_i\preceq
%\pi b_i a_i$ and  $\pi \preceq \pi a_i\preceq \pi a_i b_{i+1}$ are
%maximal chains. Then, by relation $(B8)$, $\pi b_i a_i =\pi \pi
%a_i b_{i+1}$, thus $\pi a_i\vee \pi b_i= \pi b_i a_i=\pi a_i b_{i+1}$,
%$\alpha(b_i,a_i)=b_i,a_i$ and $\alpha(a_i,b_i)=a_i,b_{i+1}$.
If $\pi a_i$ and $\pi b_i$ cover $\pi$ (for some $1 \le i \le n-1$) then, by Corollary~\ref{t.cover},
$\pi(i) > 0$, $\pi(i+1) < 0$ and $|\pi(i)| < |\pi(i+1)|$.
Then $\pi a_i \vee \pi b_i= \pi b_i a_i=\pi a_i b_{i+1}$, and again
$\alpha(a_i, b_i) = a_i b_{j+1}$ and $\alpha(b_i, a_i) = b_i a_i$ are unique and independent of $\pi$.

\smallskip

Similarly, if $|i-j|>1$ then
%%$\pi\preceq  \pi a_i\Longleftrightarrow \pi a_j\preceq \pi a_j a_i$,
% thus \\ $\pi a_i\vee \pi a_j=\pi a_i a_j$ and $\alpha(a_i,a_j)=a_i, a_j$.
$\pi a_i$ and $\pi a_j$ cover $\pi$ if and only if
$\pi(i+1) < 0$, $|\pi(i)| < |\pi(i+1)|$, $\pi(j+1) < 0$ and $|\pi(j)| < |\pi(j+1)|$.
Then $\pi a_i\vee \pi a_j = \pi a_i a_j$, so that
$\alpha(a_i, a_j) = a_i a_j$ is clearly unique and independent of $\pi$.

\smallskip

In all of the above cases, the interval $[\pi, \pi s \vee \pi s']$ is of length $2$.
It is easy to identify which generators $a_k$ appear in a maximal chain,
and the rest readily follows.

\smallskip

Let us now turn to the last case, and assume that $\pi a_i$ and $\pi a_{i+1}$ cover $\pi$,
for some $1 \le i \le n-2$.
By Corollary~\ref{t.cover},
$\pi(i+1) < 0$, $\pi(i+2) < 0$ and $|\pi(i)| < |\pi(i+1)| < |\pi(i+2)|$.
By Lemma~\ref{t.meet_join}(ii), $\sigma := \pi a_i \vee \pi a_{i+1}$ satisfies
$|\sigma| = |\pi| s_i s_{i+1} s_i$ and
$\Pos(\sigma^{-1}) = \Pos(\pi^{-1}) \cup \{|\pi(i+1)|, |\pi(i+2)|\}$.
Thus $\finv(\sigma) - \finv(\pi) = 2 \cdot 3 - 2 =4$.
The only maximal chains in $S_n$ from $|\pi|$ to $|\sigma|$ correspond to
$s_i s_{i+1} s_i$ and $s_{i+1} s_i s_{i+1}$, and thus
each maximal chain from $\pi$ to $\sigma$ must correspond to either
$a_i a_{i+1} a_i$ or $a_{i+1} a_i a_{i+1}$, with one additional letter of type $b_k$.
It is easy to see that the only possibilities are
$\alpha(a_i, a_{i+1})= a_i a_{i+1} b_{i+1} a_i$ and
$\alpha(a_{i+1}, a_i)= a_{i+1} b_{i+1} a_i a_{i+1}$.

%Then, by Corollary~\ref{t.cover},  $i+1,i+2\in \Neg(\pi)$ and
%$|\pi(i)|>|\pi(i+1)|>|\pi(i+2)|$. Combining this with
%Proposition~\ref{t.lattice}, $|\pi a_i \vee \pi a_{i+1}|$ is
%obtained from $|\pi|$ by a right multiplication with the
%transposition $(i,i+2)$, equivalently $|\pi a_i \vee \pi
%a_{i+1}|=|\pi a_i a_{i+1} a_i|$, and $\Neg(\pi a_i \vee \pi
%a_{i+1})=\Neg(\pi) \setminus \{i,i+1\}$. Thus
%$$
%\pi a_i \vee \pi a_{i+1}=\pi  a_i a_{i+1} b_{i+1} a_i=\pi a_{i+1} b_{i+1} a_i a_{i+1}.
%$$

%By Corollary~\ref{t.cover}, it is easy to verify that both
%$$\pi\preceq \pi a_i \preceq \pi a_i a_{i+1}  \preceq \pi  a_i
%a_{i+1}  b_{i+1} \preceq \pi a_i a_{i+1} b_{i+1} a_i=\pi a_i \vee
%\pi a_{i+1}$$ and $$\pi\preceq \pi  a_{i+1} \preceq \pi  a_{i+1} b_{i+1}
% \preceq \pi  a_{i+1} b_{i+1}  a_i \preceq \pi  a_{i+1}
%b_{i+1} a_i a_{+1}=\pi a_i \vee \pi a_{i+1}$$ are maximal chains.
%One concludes that
%$$\alpha(a_i,a_{i+1})= a_i, a_{i+1}, b_{i+1},
%a_i $$ and $$\alpha(a_{i+1}, a_i)= a_{i+1}, b_{i+1}, a_i,
%a_{i+1}.$$

%\end{proof}
\qed

\medskip

\noindent
{\it Proof of Proposition~\ref{t.Tits}.}
The proof is similar to the  analogous proof for the symmetric
group~\cite[ Theorem 3.3.1]{BB},
and proceeds by induction on the difference
between the ranks of the top and bottom elements.

%We will prove a stronger statement: For any pair $\pi \preceq
%\sigma$ in $(B_n, \preceq)$, any two labelled maximal chains in
%the interval $[\pi,\sigma]$ are connected via the moves
%$(T1)-(T5)$. This will be proved by induction on the difference
%between the ranks of the top and bottom elements.
%$\finv(\sigma)-\finv(\pi)$.

If the difference is zero then the statement obviously holds.

Assume that the difference is $k>0$. Consider two %different
maximal chains in the interval $[\pi,\sigma]$, corresponding to the words
$s s_2 \cdots s_k$ and $s' s_2' \cdots s_k'$; all letters are in $S_{2,n}$.
Thus
\[
\sigma=\pi s s_2 \cdots s_k = \pi s' s_2'\cdots s_k'.
\]

If $s = s'$ then the statement holds by the induction hypothesis for the interval $[\pi s, \sigma]$.

If $s\ne s'$ then $\pi s\preceq \sigma$ and $\pi s'\preceq \sigma$.
By the lattice property, $\pi s\vee \pi s'\preceq \sigma$.
%Recall $\alpha(s,s')$ be a labelling of a maximal chain
%from $\pi$ to $\pi s \vee \pi s'$ which starts with $s$ (such a
%labelling always exist by Lemma~\ref{t.2-atoms}).
By Lemma~\ref{t.2-atoms} there exists a maximal chain in the interval $[\pi,\pi s\vee \pi s']$
corresponding to the word $\alpha(s, s')$ starting with $s$.
It can be extended to a maximal chain in $[\pi, \sigma]$ corresponding to the word $\alpha(s, s')\beta$,
where $\beta$ corresponds to some maximal chain in $[\pi s\vee \pi s', \sigma]$.
Both words $s s_2 \cdots s_k$ and $\alpha(s, s')\beta$ start with $s$.
By the induction hypothesis for $[\pi s, \sigma]$,
it is possible to transform $s s_2 \cdots s_k$ into $\alpha(s, s')\beta$ using the moves $(T1)-(T5)$.
By the same argument for $s'$,
it is possible to transform $\alpha(s', s)\beta$ into $s' s_2' \cdots s_k'$ using the moves $(T1)-(T5)$.
Finally, by Lemma~\ref{t.2-atoms},
%for all pairs $s,s'\in S_{2,n}$, with $\pi \preceq \pi s$ and $\pi\preceq \pi s'$ for some $\pi\in B_n$,
it is possible to transform $\alpha(s,s')\beta$ into $\alpha (s',s)\beta$ using one of the moves $(T1)-(T5)$,
thus completing the proof.

\qed

\section{Bivariate Distribution}
%\ \\
Let
\[
E_n(t) := \sum\limits_{\pi\in S_n} t^{\des(\pi)}
\]
be the {\em Eulerian Polynomial}. More generally, let
\[
S_n(q,t) := \sum\limits_{\pi\in S_n} q^{\inv(\pi)}t^{\des(\pi)}.
\]
Recall that $(B_n,\preceq)$ is graded by $\finv$.

\begin{defn}\label{df-wdes}
For every $\pi\in B_n$ let $\wdes(\pi)$ be the number of elements
in $B_n$ which are covered by $\pi$ in the poset $(B_n,\preceq)$.
\end{defn}

%Clearly, for $r=1$ $\wdes=\des$.

\begin{lemma}\label{wdes}
For every $\pi\in B_n$
$$
\wdes(\pi)=\#\left( \D(|\pi|)\cup \Neg(\pi) \right).
$$
\end{lemma}

\begin{proof}
By Corollary~\ref{t.cover} (with $\pi$ and $\sigma$ interchanged),
$\sigma$ is covered by $\pi$ in $(B_n, \preceq)$ if and only if
\begin{itemize}
\item[(i)] there exists $1\le i\le n$, such that
$$
i\in \Neg(\pi) \ \ \ \ \ \ \ \ {\rm{and}}\ \ \ \ \ \ \ \
\sigma=\pi b_i;
$$
or \item[(ii)] there exists $1 \le i \le n-1$, such that
$$
i\not\in \Neg(\pi), \ \ \ \ \ |\pi(i)|>|\pi(i+1)| \ \ \ \ \ \ \ \
{\rm{and}}\ \ \ \ \ \ \ \ \sigma=\pi a_i^{-1}.
$$
\end{itemize}
Hence, the set of elements which are covered by $\pi$ in $(B_n,
\preceq)$ is
$$
\{\pi b_i:\ i\in \Neg(\pi)\} \cup \{\pi a_i^{-1}:\ i\in
\D(|\pi|)\setminus \Neg(\pi)\},
$$
a disjoint union.

It follows that
$$
\wdes(\pi)= \#\Neg(\pi)+ \#(\D(|\pi|)\setminus\Neg(\pi)) =
\# (\D(|\pi|)\cup \Neg(\pi)).
$$

\end{proof}

\begin{proposition}
For every $n$,
\[
\sum\limits_{\pi\in B_n} t^{\wdes(\pi)}=
(1+t)^n\cdot E_n \left(\frac{2t}{1+t}\right)
\]
and
\[
\sum\limits_{\pi\in B_n} q^{\finv(\pi)} t^{\wdes(\pi)}=
(1+qt)^n\cdot S_n \left(q^2,\frac{(1+q)t}{1+qt}\right).
\]
\end{proposition}

\begin{rem}
By a well known result of Stanley~\cite{Stanley79},
$S_n(q,t)$ has an elegant $q$-exponential generating function.
%, see e.g.~\cite[??]{BB},
It follows that the same is true when the pair $(\finv, \wdes)$ is used instead of $(\inv, \des)$.
\end{rem}

\begin{proof}
%We will prove the second identity. The first identity is obtained
%by substituting $t=1$ and $q=t$.
$\ZZ_2^n$ and $S_n$ can be viewed as subgroups of $B_n$,
restricting elements $\pi \in B_n$ to have $|\pi| = id$ or $\pi(i) > 0$ ($\forall i$), respectively.
Moreover, every $\pi\in B_n$ can be written in the form $\pi=vu$ for some $u\in \ZZ_2^n$ and
$v=|\pi|\in S_n$. Hence
\[
\sum\limits_{\pi\in B_n} t^{\wdes(\pi)} =
\sum\limits_{u\in \ZZ_2^n} \sum\limits_{v\in S_n} t^{\wdes(vu)}.
\]
By Lemma~\ref{wdes}, the right hand side is equal to
\[
\sum\limits_{u\in \ZZ_2^n} \sum\limits_{v\in S_n} t^{\# (\D(v)\cup \Neg(u))} =
\sum\limits_{v\in S_n} \sum\limits_{u\in \ZZ_2^n}
t^{\# \D(v)} t^{\# (\Neg(u)\setminus \D(v))}
\]
\[
= \sum\limits_{v\in S_n} t^{\# \D(v)} \sum\limits_{u \in \ZZ_2^n} t^{\# (\Neg(u)\setminus \D(v))}
= \sum\limits_{v\in S_n} t^{\# \D(v)} 2^{\#\D(v)} (1+t)^{n-\#\D(v)}
\]
\[
= (1+t)^n\sum\limits_{v\in S_n} \left( \frac{2t}{1+t} \right)^{\# \D(v)}
= (1+t)^n\cdot E_n \left(\frac{2t}{1+t}\right).
\]
The proof of the second identity is similar.
\begin{eqnarray*}
\sum\limits_{\pi\in B_n} q^{\finv(\pi)} t^{\wdes(\pi)}
& = & \sum\limits_{u\in \ZZ_2^n} \sum\limits_{v\in S_n} q^{\finv(vu)} t^{\wdes(vu)}\\
& = & \sum\limits_{u\in \ZZ_2^n} \sum\limits_{v\in S_n} q^{2\cdot \inv(v)+\# \Neg(u)}t^{\# (\D(v)\cup \Neg(u))}\\
& = & \sum\limits_{v\in S_n} q^{2\cdot \inv(v)} t^{\# \D(v)}
\sum\limits_{u\in \ZZ_2^n}q^{\#\Neg(u)} t^{\# (\Neg(u)\setminus \D(v))}\\
& = & \sum\limits_{v\in S_n} q^{2\cdot \inv(v)} t^{\# \D(v)} (1+q)^{\#\D(v)} (1+qt)^{n-\#\D(v)}\\
& = & (1+qt)^n\sum\limits_{v\in S_n} q^{2 \cdot \inv(v)} \left(\frac{(1+q)t}{1+qt}\right)^{\# \D(v)}\\
& = & (1+qt)^n\cdot S_n \left(q^2,\frac{(1+q)t}{1+qt}\right).
\end{eqnarray*}

\end{proof}

%\section{Maximal Chains vs. Reduced Words}

%\subsection{Enumeration}

%and Stanley symmetric functions ???

\section{Wreath Products}\label{wreath}

The above results generalize to the group $G(r,n):=\ZZ_r\wr S_n$,
for every positive integer $r$. Proofs are similar and will be left to the reader.

%What about general $r$ ?

%Denote $G(r,n):=\ZZ_r\wr S_n$ and $\omega:=e^{2\pi i/r}$.
%%The
%%classical Coxeter generating set of $G(r,n)$ does not reflect the
%%fact that $G(r,n)$ is a semi-direct product of $\ZZ_2^n$ with
%%$S_n$, i.e. that $\ZZ_r^n$ is a normal subgroup of $G(r,n)$ and
%%$S_n$ is the corresponding quotient.

For $1\le i\le n$ define the vector $d_i := (\delta_{i1}, \ldots, \delta_{in}) \in \ZZ_r^n$,
where
\[
\delta_{ij} =
\begin{cases}
1, & \text{if $i=j$;}\\
0, & \text{otherwise.}
\end{cases}
\]
Let
\[
%a_i:= (d_{i+1}, s_i  ) \qquad (1 \le i \le n-1)   %grn
a_i:= (d_i, s_i) \qquad (1 \le i \le n-1)
\]
and
\[
%b_i:=(d_i, id) \qquad (1 \le i \le n).  %grn
b_i:=(d_i, id) \qquad (1 \le i \le n).
\]
%Note that we use $d_{i+1}$ for $a_i$ since, by our convention from Section~\ref{s.prelim},
%in an element of $G(r,n)$ the colors are applied ``after'' the permutation.
%On the other hand, the definition of $\Neg$ for an element of $B_n$ assumes that the sign is applied
%``before'' the permutation, and this should be taken into account when results are translated
%from $B_n$ to $G(r,n)$.

\begin{proposition}\label{t.relation1-r}
The wreath product $G(r,n)=\ZZ_r\wr S_n$ is %isomorphic to the group
generated by the set $S_{r,n}:=\{a_i\,:\, 1 \le i \le n-1\} \cup \{b_i\,:\, 1 \le i \le n\}$
with defining relations  $(B1) - (B8)$ of Proposition~\ref{t.relation1}, except that %where
relation $(B1)$ is replaced by
\[
b_i^r=1 \qquad (1 \le i \le n), \leqno(B1_r)
\]
\end{proposition}

\medskip

Recall Definition~\ref{finv} of the flag inversion number.
%, which was introduced by Foata and Han~\cite{FH1, FH}: for %an $r$-colored permutation $\pi\in G(r,n)$ let
%$$
%\finv(\pi):=r\cdot\inver (|\pi|)+ \sum\limits_{i=1}^n c_i.
%$$
%Then

\begin{defn}\label{weak-definition-r}
The {\em flag (right) weak order} on $G(r,n)$, $\preceq$, is the
reflexive and transitive closure of
%$\pi\preceq\cdot\ \sigma$ if $\sigma=\pi s$ for some $s\in S_{r,n}$ and $\finv(\pi)<\finv(\sigma)$.
the relation
%$\hat\preceq$ if $\sigma=\pi s$ for some $s\in S_{2,n}$ and $\finv(\pi)<\finv(\sigma)$.
$$\pi \lessdot %i\hat\preceq\
\pi s\ \ \ \ \  \Longleftrightarrow\ \ \ \ \  \pi\in G(r,n), s \in
S_{r,n}\ \ \ \text{ and }\ \ \ \finv(\pi)<\finv(\pi s).$$
\end{defn}

\medskip

\begin{proposition}\label{t.basic-r}
The poset $(G(r,n),\preceq)$ is
\begin{itemize}
\item[(i)]
ranked (by flag inversion number);

\item[(ii)]
self-dual (by $\pi\mapsto \bar\pi \mu_0$, where
$\mu_0 = ((r-1, \ldots, r-1), [n, \ldots, 1])$
is the unique maximal element in this order and
$\bar\pi = ((-c_1, \ldots, -c_n), \tau)$ when
$\pi = ((c_1, \ldots, c_n), \tau)$);\ and

\item[(iii)]
rank-symmetric and unimodal.

\end{itemize}
\end{proposition}

%\medskip

\begin{proposition}\label{t.cover-r}
$\sigma$ covers $\pi$ in $(G(r,n), \preceq)$ if and only if either
\begin{itemize}
\item[(i)]
there exists $1\le i\le n$ such that
\[
c_i(\pi) \ne r-1 \ \ \ \ \ \ {\rm{and}}\ \ \ \ \ \  \sigma=\pi b_i;
\]
or
\item[(ii)]
there exists $1 \le i \le n-1$ such that
$$
c_{i+1}(\pi) = r-1, \ \ \ \ \ |\pi(i)|<|\pi(i+1)| \ \ \ \ \ \  {\rm{and}}\ \ \ \ \ \ \sigma=\pi a_i.
$$
\end{itemize}
\end{proposition}

%For $\pi=((c_1,\dots,c_n), |\pi|)\in G(r,n)$ let
%$\N(\pi):=(c_1,\dots,c_n)\in [0,r-1]^n$.

\noindent
In the following statement, elements $-c_j(\pi^{-1}) = c_{|\pi^{-1}(j)|}(\pi) \in \ZZ_r$
are compared using the natural linear order $0 < 1 < \ldots < r-1$ on $\ZZ_r$.

\begin{proposition}\label{t.criterion-r}
For every $\pi, \sigma\in G(r,n)$,
\begin{eqnarray*} %\label{e.preceq-r}
\pi \preceq \sigma &\Longleftrightarrow&
\Inv(|\pi^{-1}|) \subseteq \Inv(|\sigma^{-1}|) {\text{\rm \ \ and \ }}\\
\nonumber & & \{j \,:\, -c_j(\pi^{-1}) > -c_j(\sigma^{-1})\} \subseteq
\M\left[\Inv(|\sigma^{-1}|) \setminus \Inv(|\pi^{-1}|)\right].
\end{eqnarray*}
\end{proposition}

%\begin{proposition}\label{t.criterion-r}
%For every $\pi, \sigma\in G(r,n)$, $\pi\preceq \sigma$ if and only if
%\begin{itemize}
%\item[1.]  $\N(\pi)_i> \N(\sigma)_i \Longrightarrow i\in
%\M(\Inv(|\sigma^{-1}|)\setminus \Inv(|\pi^{-1}|))$  for every
%$1\le i\le n$,\ \ and \item[2.] $\Inv(|\pi^{-1}|)\subseteq
%\Inv(|\sigma^{-1}|)$.
%\end{itemize}
%\end{proposition}

\noindent
It follows that all the results of Section~\ref{s.weak} can be generalized
to $G(r,n)$. In particular,

\begin{proposition}\label{t.lattice-r}
The poset $(G(r,n),\preceq)$ is a %complemented
lattice.
\end{proposition}

%\medskip
%\noindent
%In the following explicit description of meets and joins, 
%minima and maxima of sets of elements of $\ZZ_r$ are taken
%with respect to the linear order $0 < 1 < \ldots < r-1$.

\begin{lemma}\label{t.meet_join-r}
Let $A$ be an arbitrary subset of $G(r,n)$.
%, let $\tau$ be the meet of the corresponding unsigned permutations in the weak order on $S_n$,
%that is $%\left|\bigwedge_{\sigma\in A} \pi \right| \\tau:= \bigwedge_{\sigma\in A} |\sigma|.$
%Then
%For every subset $A\subseteq B_n$
\begin{itemize}
\item[(i)]
The meet $A_\wedge$ of $A$ in $(G(r,n),\preceq)$ is determined by
\[
|A_\wedge| := \bigwedge_{\sigma\in A} |\sigma|,
\]
where the meet is taken with respect to the (right) weak order on $S_n$, and by
\[
%-c_j(A_\wedge^{-1}) := \min_{\sigma \in A} m_j^\wedge(\sigma) \qquad(1 \le j \le n),
-c_j(A_\wedge^{-1}) := 
\min \{-c_j(\sigma^{-1}) \,:\, \sigma \in A,\,  j \not\in M(|A_\wedge|, |\sigma|)\}\qquad(1 \le j \le n),
\]
where $M(|\pi|, |\sigma|) := \M\left[\Inv(|\sigma^{-1}|) \setminus \Inv(|\pi^{-1}|)\right]$ and
the minimum is taken with respect to the linear order $0 < 1 < \ldots < r-1$ on $\ZZ_r$,
using the convention $\min \emptyset := \max \ZZ_r = r-1$.
%where the minimum is taken with respect to the linear order $0 < 1 < \ldots < r-1$ on $\ZZ_r$ 
%and
%\[
%m_j^\wedge(\sigma) :=
%\begin{cases}
%   r-1,  & \text{ if } j \in M(|A_\wedge|, |\sigma|);  \\
%   -c_j(\sigma^{-1}), & \text{ otherwise. }
%\end{cases}
%\]

\item[(ii)]
The join $A_\vee$ of $A$ in $(G(r,n),\preceq)$ is determined by
\[
|A_\vee| := \bigvee_{\sigma \in A} |\sigma|
\]
%where the join is taken with respect to the (right) weak order on $S_n$, 
and by
\[
%-c_j(A_\vee^{-1}) := \max_{\sigma \in A} m_j^\vee(\sigma) \qquad(1 \le j \le n),
-c_j(A_\vee^{-1}) := 
\max \{-c_j(\sigma^{-1}) \,:\, \sigma \in A,\,  j \not\in M(|\sigma|, |A_\vee|)\}\qquad(1 \le j \le n),
\]
using the convention $\max \emptyset := \min \ZZ_r = 0$.
%where the maximum is taken with respect to the linear order $0 < 1 < \ldots < r-1$ on $\ZZ_r$ 
%and
%\[
%m_j^\vee(\sigma) :=
%\begin{cases}
%   0,  & \text{ if } j \in M(|\sigma|, |A_\vee|);  \\
%   -c_j(\sigma^{-1}), & \text{ otherwise. }
%\end{cases}
%\]

\end{itemize}

\end{lemma}

\begin{proposition}\label{t.homotopy-type-r}
Suppose that $\pi\prec \sigma$ in $G(r,n)$ and $\finv(\sigma)-\finv(\pi)\ge 2$.
Then the order complex of the open interval $(\pi,\sigma)$ is
homotopy equivalent to the sphere ${\bf S}^{k-2}$
if $\sigma$ is the join of $k$ atoms in the interval $[\pi, \mu_0]$,
and is contractible otherwise.
\end{proposition}

\begin{corollary}
For every $\pi, \sigma\in G(r,n)$,
$$
\mu(\pi,\sigma)=\begin{cases}
   (-1)^k,   & \text{ if } \sigma \text{ is a join of } k \text{ atoms in } [\pi, \mu_0];  \\
   0, & \text{ otherwise. }
\end{cases}
$$
\end{corollary}

\begin{defn}
For every $\pi\in G(r,n)$ let $\wdes(\pi)$ be the number of
elements in $G(r,n)$ which are covered by $\pi$ in the poset
$(G(r,n),\preceq)$.
\end{defn}

Clearly, for $r=1$ $\wdes$ is the standard descent number. For
$r=2$ it coincides with Definition~\ref{df-wdes}.

\begin{proposition}
For every $n$ and $r$,
$$
\sum\limits_{\pi\in G(r,n)} t^{\wdes(\pi)} =
(1+(r-1)t)^n \cdot E_n\left(\frac{rt}{1+(r-1)t}\right)
$$
and
$$
\sum\limits_{\pi\in G(r,n)} q^{\finv(\pi)} t^{\wdes(\pi)} =
\left( 1+[r-1]_q qt \right)^n \cdot S_n \left(q^r, \frac{[r]_q t}{1+[r-1]_q qt}\right).
$$
\end{proposition}

%\section{Final Remarks}

%A strong analoge of the flag weak order, namely, a flag strong
%order on wreath products (and on colored rook monoids) is
%suggested and studied in~\cite{ACR}, where it is shown to be
%intimately related to $(q,t)$-Stirling numbers.

\medskip

%Similar presentations of other groups will be discussed elsewhere.

\section{Final Remarks and Open Problems}

Recall the pseudo-Coxeter moves $(T1)-(T5)$ from Proposition~\ref{t.Tits}.
Consider the graph $\Gamma_n$, whose
vertices are all maximal chains in the flag weak order on $B_n$
and whose edges correspond to these moves.
%$(T1)-(T5)$.
By Proposition~\ref{t.Tits}, $\Gamma_n$ is connected.

\begin{problem}\label{prob1}
Find the diameter of $\Gamma_n$.
\end{problem}

For a solution of an analogous problem for the classical weak
orders of types $A$ and $B$ see~\cite{Reiner-R}.

\medskip

Following comments of an anonymous referee, it should be noted
that progress toward a solution of Problem~\ref{prob1} may be
obtained by explicit calculation of various poset parameters such
as order dimension and width. Another approach is
%diameter computations is % better understanding of the poset, in particular, finding
a search for symmetries induced by group actions, as
well as recursive poset properties such as supersolvabilty. Such
methods were found useful in similar contexts; see, e.g., 
\cite{TFT1, Reiner-R}.

\medskip

It is now natural to look for a definition of a nicely-behaved
weak order on other complex reflection groups. A key tool may be the discovery of
convenient presentations for kernels of one-dimensional characters.
 % of a classical complex reflection group $G(r,p,n)$.

\medskip

A challenging problem is to find a ``correct" definition of
strong (Bruhat) order %and absolute order
on wreath products and other complex reflection groups,
%One may expect this order
%%to have a suitable interval structure, to
%%respect parabolic subgroups,
%to
having desired properties (such as a nice interval
structure and a subword property) which, hopefully, demonstrate an
interplay with the flag weak order. Such an order may be useful in
developing an appropriate Kazhdan-Lusztig theory.
% for the development of
%of deformed algebras and representations and
%associated  Kazhdan-Lusztig polynomials.
%We do hope that this is not too much to expect, and that such an order does exist.

\medskip

Finally, finding an absolute order on wreath products and other
complex reflection groups %whose maximal intervals form
%an order
may provide interesting new %analogues
extensions of the non-crossing partition lattice.
%, is also desired.

\bigskip

\noindent{\bf Acknowledgements.} The authors thank the anonymous
referees for many helpful comments.


\begin{thebibliography}{99}

\bibitem{TFT1} R.\ M.\ Adin, M.\ Firer and Y.\ Roichman,
{\it Triangle-free triangulations},
Adv.\ Appl.\ Math.~{\bf 45} (2010), 77--95.

\bibitem{ACR}
R.\ M.\ Adin, Y.\ Cherniavsky and Y.\ Roichman,
{\it Strong order on colored rook monoids and $(q,t)$-Stirling numbers},
in preparation.

\bibitem{BB}
A.\ Bj\"{o}rner and F.\ Brenti,
Combinatorics of Coxeter groups,
Graduate Texts in Mathematics, 231, Springer, New York, 2005.

\bibitem{Bou}
N.\ Bourbaki,
Lie Groups and Lie Algebras,
English translation by Andrew Pressley, Springer, 2002.

\bibitem{BRR}
F.\ Brenti, V.\ Reiner and Y.\ Roichman,
{\it Alternating subgroups of Coxeter groups},
J.\ Combin.\ Theory Ser.\ A~{\bf 115} (2008), 845--877.

\bibitem{FR}
H.\ L.\ M.\ Faliharimalala and A.\ Randrianarivony,
{\it Flag-major index and flag-inversion number on colored words and wreath product},
S\'em.\ Lothar.\ Combin.~{\bf 62} (2010), Art.\ B62c, 10 pp.\ (electronic).

\bibitem{Fire}
M.\ Fire,
{\it Statistics on wreath products},
preprint, {\tt arXiv:math/0409421}.

\bibitem{FH1}
D.\ Foata and G.\ -N.\ Han,
{\it Signed words and permutations. I. A fundamental transformation},
Proc.\ Amer.\ Math.\ Soc.~{\bf 135} (2007), 31--40.

\bibitem{FH}
D.\ Foata and G.\ -N.\ Han,
Statistical Distributions on Words and $q$-Calculus on Permutations,
Strasbourg, 2007.

\bibitem{Humphreys}
J.\ E.\ Humphreys,
Reflection groups and Coxeter groups,
Cambridge Studies in Advanced Mathematics, 29,
Cambridge University Press, Cambridge, 1990.

\bibitem{Reiner-R}
V.\ Reiner and Y.\ Roichman,
{\it Diameter of reduced words},
Trans.\ Amer.\ Math.\ Soc., to appear. {\tt arXiv:0906.4768}.

\bibitem{Stanley79}
R.\ P.\ Stanley,
{\it Binomial posets, M\"obius inversion, and permutation enumeration},
J.\ Combin.\ Theory Ser.\ A~{\bf 20}
(1976), 336--356.

\bibitem{VV}
A.\ Vershik and M.\ Vsemirnov,
{\it The local stationary presentationof the alternating groups and the normal form},
J.\ Algebra~{\bf 319} (2008), 4222--4229.

\end{thebibliography}
\end{document}